\documentclass[12pt,a4paper]{amsart}
\usepackage{setspace}
\usepackage{amscd}
\usepackage{amsmath}
\usepackage{latexsym}
\usepackage{amsfonts}
\usepackage{amssymb}
\usepackage{amsthm}
\usepackage{graphicx}
\usepackage{verbatim}
\usepackage{mathrsfs}
\usepackage{enumerate}
\usepackage{color}
\newtheorem{theorem}{Theorem}[section]
\newtheorem{lemma}[theorem]{Lemma}

\newtheorem{proposition}[theorem]{Proposition}
\theoremstyle{remark}
\newtheorem{remark}[theorem]{Remark}
\newcommand{\RR}{\mathbb{R}}
\newcommand{\NN}{\mathbb{N}}
\newcommand{\la}{\lambda}

\newcommand{\dd}{\mathrm{ \,d}}
\newcommand{\e}{\varepsilon}
\newcommand{\uloc}{\textnormal{uloc}}

\newcommand{\Rmnum}[1]{\expandafter\@slowromancap\romannumeral #1@}
\allowdisplaybreaks
\begin{document}
\title[] { Global-in-time Boundedness of solution for  Cauchy problem to  the Parabolic-Parabolic Keller-Segel system with logistic growth}
\author[Y. Nie and  X. Zheng]{Yao Nie and Xiaoxin Zheng}
\address{Institute of Applied Physics and Computational
Mathematics, Beijing 100191, P.R. China}
\email{nieyao930930@163.com}
\address{School of Mathematical Sciences, Beihang University, Beijing 100191, P.R. China}
\email{xiaoxinzheng@buaa.edu.cn}
\date{\today}
\begin{abstract} We study global-in-time well-posedness and the behaviour and of the solution to Cauchy problem in the classical Keller-Segel system with logistic term
\begin{equation*}
\left. \aligned
 \partial_tn-\Delta n=&-\chi\nabla\cdot(n\nabla c)+\la n-\mu n^2 \\
 \tau\partial_tc-\Delta c=&-c+n
\endaligned
\right\}\quad\text{in}\,\,\,\RR^d\times\RR^+,
\end{equation*}
where $d\ge 1$, $\tau,\, \chi,\, \mu>0$ and $\lambda\ge 0$. It's inspired by a previous result \cite[M. Winkler, Commun. Part. Diff. Eq., 35 (2010), 1516-1537]{Win10}, where the global-in-time boundedness of the above Keller-Segel system in smooth \emph{bounded }convex domains is established for large $\mu$. However, his approach in bounded domain ceases to directly apply in the entire space $\RR^d$, and then  they raised an interesting question whether a similar global-in-time boundedness statement remains true of Cauchy problem. In this paper, we answer this open problem by developing local-in-space estimates. More precisely, we prove that the above Keller-Segel system possesses a uniquely global-in-time bounded solution for any $\tau>0$ under the assumption that $\mu$ is large. The key point of our proof  heavily relies on localization in space of solution caused by ``local effect" of $L^\infty(\RR^d)$-norm.
\end{abstract}
\maketitle
\section{Introduction}
 \textit{Chemotaxis}, an extremely universal biological phenomenon, the biased movement of biological individuals (e.g., body cells, bacteria, single cells and multicellular organisms) in response to certain chemical gradients in the surrounding environment. Chemotaxis performs a pivotal role throughout the life cycle of biological individuals. For example, cancer cells secrete some enzymes for more cancerous cells to invade \cite{CSS05}, the egg sends some chemical substances to attract sperm cells \cite{HK05}, and immune cells move to sites of  inflammation \cite{Wu05}. This attracts many researchers to conduct extensive experiments into chemotactic mechanism of bacteria, such as \emph{Escherichia coli }\cite{BWS06} and \emph{Dictyostelium discoideum} \cite{DW06}. The mathematical modeling of chemotactic processes can be traced back to the pioneering works of Patlak in~\cite{Pa53} and Keller and Segel in \cite{KS70, KS71}. Nowadays, Keller-Segel system is widely known as a chemotaxis model and  its most prototypical version  coupling two parabolic equations is given by
 \begin{equation}\label{KS}
\left.\aligned
& \partial_tn-\Delta n=-\chi\nabla\cdot(n\nabla c)\\
 &\tau\partial_tc-\Delta c=-c+n \\
\endaligned
\right\},
\end{equation}
where $n$ and $c$ respectively denote the density of cells and the concentration of chemical substances. $\chi>0$, chemotactic coefficient,  measures the strength of attractive chemotaxis. $\tau> 0$ represents a relaxation time scale.
There exists a vast mathematical literature on system \eqref{KS},  mainly focusing on existence, boundedness, large time behaviour of bounded solutions and blow-up mechanisms. One striking feature of equations \eqref{KS} consists in its ability to chemotactic collapse under certain circumstances, shown
in many experimental frameworks and in rigorous mathematical framework of singularity formation (the $L^\infty$ norm of the cell density $n(x,t)$ becomes unbounded when $t$ reaches a blow-up time). Full of biological significance and mathematical challenges, many literatures are devoted to this popular issue on whether global-in-time bounded solutions exist, and if not, when blow-up phenomenon occurs. Some well-known results corresponding to this problem for full parabolic-parabolic system \eqref{KS} are listed as follows: In 1-D,  all solutions of equations \eqref{KS} are global and bounded, which implies chemotactic aggregation is entirely ruled out \cite{HP04, OY01}. In 2-D, one can find a threshold such that if initial mass is lower than the threshold, the solution will be global and bounded in a bounded domain with homogeneous Neumann boundary condition \cite{GZ98,NSY97} and global exist in $\RR^2$ \cite{CC08}. And there exist initial data such that the corresponding solution blows up either in finite or infinite time, provided initial mass ($L^1$ norm of  bacteria) is larger than a threshold under homogeneous Neumann boundary conditions in a smooth bounded domain \cite{HW01, SS01}. In higher dimension $d\ge 3$, smallness of initial mass cannot efficiently prevent chemotactic collapse, this conclusion has been proved by constructing blow-up solutions in finite time if domain is a ball \cite{Win10JDE, Win13} and in entire space~$\RR^d$~\cite{Win20}. Instead, smallness condition of initial data $(n_0, c_0)$ in $L^{\frac{d}{2}}\times \dot W^{1,d}$ under Neumann boundary conditions in a bounded domain \cite{Cao15} and in $L^{a}\times \dot W^{1,d}$ for $\frac{d}{2}<a\le d$ in the whole space \cite{CP08} guarantee global existence of solutions.
In many biological phenomena, the cell growth and death play an indelible role in chemotactic movement on small timescales, and the aggregating pattern-dynamics involved in growth of the population was first studied in \cite{MT96}. A prototypical choice to reveal this cell kinetics is to add a logistic source
$\lambda n-\mu n^2$ in the evolution equation for $n$ in \eqref{KS}. Here $\lambda\ge0, \mu>0$. Hence, the fully parabolic version ($\tau>0$) of classical Keller-Segel system with logistic source is given by
 \begin{equation}\label{KSL}
\left. \aligned
 & \partial_tn-\Delta n=-\chi\nabla\cdot(n\nabla c)+\lambda n-\mu n^2\\
 &\tau\partial_tc-\Delta c=-c+n
\endaligned
\right\}
\end{equation}
As several literatures have revealed, logistic source can suppress chemotactic aggregation in some cases. For instance, solutions of equations \eqref{KSL} are global-in-time (and bounded for Neumann initial-boundary value problem) for any $\mu>0$ in the setting of dimension $d\le 2$ \cite{ OT02, OY01, OY02, Xiang15}.
Unfortunately, in higher-dimension $d\ge 3$, aggregation vs. logistic damping effect becomes complicated. And results when $d\ge 3$ mostly cover cases that $\mu$ is greater than some positive constant under the Neumann boundary condition in a \emph{bounded} domain. More precisely, Winkler \cite{Win10} firstly rigorously proved that system \eqref{KSL} possesses a unique global-in-time bounded classical solution if $\mu$ is large enough and initial data is sufficiently smooth. Subsequently, a large body of works are devoted to study qualitatively and quantitatively on $\mu$ when logistic damping suppresses bacteria aggregation. So far, it is only known that  solutions  are global and bounded  when $\mu>\theta_0\chi$ for some positive constant $\theta_0$ in Neumann problems for system \eqref{KSL} $(\forall\tau>0)$ in a $d \,(d\ge 3)$ dimensional bounded domain \cite{TW15, Xi18, Xiang18, YCJZ15} or in   Cauchy problem for system \eqref{KSL} $(\tau=1)$ in \cite{Win10}. However, for $\mu\le \theta_0\chi$ or even for small $\mu>0$,  whether or not blow-up occurs in system \eqref{KSL} is still an open problem.
As mentioned above, in high dimension ($d\ge 3$),  previous results of global boundedness for solutions of system~\eqref{KSL} are corresponding to Neumann problems in a \emph{bounded} domain. And only in the special case $\tau=1$, one so far can derive a similar result for Cauchy problem. Indeed, the function $z:=\frac{\tau}{2}|\nabla c|^2+\frac{1}{\chi}n$, as is shown in \cite{Win10, Win14}, satisfies a scalar parabolic inequality which leads to the upper bound of both $n$ and $\nabla c$. More precisely, from system \eqref{KSL} we formally infer that
\begin{align*}
\frac{\dd}{\dd t}\left(\frac{\tau}{2}|\nabla c|^2+\frac{1}{\chi}n\right)=\nabla c\cdot\nabla\Delta c-|\nabla c|^2+\frac{1}{\chi}\Delta n-n\Delta c+\frac{\lambda}{\chi}n-\frac{\mu}{\chi}n^2.
\end{align*}
Noting the fact that $\nabla c\cdot\nabla\Delta c=\Delta(\frac{1}{2}|\nabla c|^2)-|D^2 c|^2$ and $|\Delta c|^2\le d|D^2 c|^2$, it yields by Young's inequality that
\begin{align*}
\frac{\dd}{\dd t}\left(\frac{\tau}{2}|\nabla c|^2+\frac{1}{\chi}n\right)\le \Delta\left(\frac{1}{2}|\nabla c|^2+\frac{1}{\chi}n\right)-|\nabla c|^2-\left(\frac{\mu}{\chi}-\frac{d}{4}\right)n^2+\frac{\lambda}{\chi}n.
\end{align*}
Here we see that if $\tau=1$ and $\mu>\frac{d}{4}\chi$, the function $z$ satisfies that
\begin{equation}\label{z}
z_t-\Delta z +z\le \frac{(\lambda+1)^2}{4\mu\chi-d\chi^2}.
\end{equation}
Hence, in the special case $\tau=1$, one can take advantage of the above nice inequality and the maximum principle to show global boundedness of solutions when $u_0$ and $\nabla v_0$ are bounded for  Cauchy problem. Unfortunately, in the general case $\tau\neq1$, the nice structure \eqref{z} is broken. For a convex bounded domain, Winkler \cite{Win10} cleverly made use of the coupling structure of system \eqref{KSL} and exploited a Moser-type iteration to derive the global boundedness solutions for any $\tau>0$. In problems posed in the whole space $\RR^d$,
the approaches in \cite{Win10} and other relevant  literatures focused on bounded domain cease to directly apply. Hence Winkler \cite{Win10} raised an interesting question \emph{whether a similar statement remains true when the physical domain is the entire space $\RR^d$},  which is still an open problem.
The purpose of this paper is to study the above interesting problem that global boundedness for solutions of \eqref{KSL} in the whole space. To be more precise, we consider Cauchy problem of the following parabolic-parabolic Keller-Segel system with logistic source:
\begin{equation}\label{che}
\left. \aligned
 \partial_tn-\Delta n=&-\chi\nabla\cdot(n\nabla c)+\la n-\mu n^2 \\
 \tau\partial_tc-\Delta c=&-c+n
\endaligned
\right\}\quad\text{in}\,\,\,\,\RR^d\times\RR^+,
\end{equation}
\begin{equation}\label{che-I}
n(0, x)=n_0(x),\quad c(0, x)=c_0(x)\quad\text{in}\,\,\,\,\RR^d.
\end{equation}
And initial data $(n_0, c_0)$ is equipped with the conditions that
\begin{equation}\label{ini}
\left\{ \aligned
&n_0\in  L^\infty(\RR^d)\,\text{ is nonnegative},\\
&c_0\in W^{1,\infty}(\RR^d) \, \text{ is nonnegative},
\endaligned
\right.
\end{equation}
 Our main result read as follows.
\begin{theorem}\label{result}
Let $\chi,\,\tau, \,\mu>0$, $\lambda\ge 0$ and $d\ge 1$.  Assume the initial data $(n_0, c_0)$ satisfy condition \eqref{ini}. Then there exists $\mu_0>0$ such that for all $\mu\ge\mu_0$, problem \eqref{che}-\eqref{che-I} possesses a unique nonnegative classical global-in-time solution $$(n, c)\in C_{\textnormal{w}}([0,\infty);L^\infty(\RR^d))\times C_{\textnormal{w}}([0,\infty);W^{1,\infty}(\RR^d))$$ which enjoys global-in-time boundedness, that is, there exits a positive constant $C $ only depending on $\chi,\,\tau, \,\lambda,\, \mu,\, d$, $\|n_0\|_{  L^\infty(\RR^d)}$ and $\|c_0\|_{W^{1,\infty}(\RR^d)}$ such that for all $t>0,$
\[\|n(t)\|_{L^\infty(\RR^d)}+\|  c(t)\|_{W^{1,\infty}(\RR^d)}\le C.\]
\end{theorem}
\begin{remark}
 Compared with Winkler's result \cite{Win10}, we get the same result in Theorem \ref{result} for Cauchy problem, but the method is very different.
 To overcome difficulty caused by the unbounded domain $\RR^d$, we develop localization in space of solution in order to exploit damping effect  of logistic source, which  helps in turn to establish global-in-time boundedness of solution. And we believe that our method here can be used the other related model in the unbounded domain with logistic source.
\end{remark}
\begin{remark}
For $d=1,2$, we can remove  the additional condition $\mu\ge\mu_0$ as in \cite{Win10} for the initial boundary value problem.
\end{remark}
Let us  now sketch the  main idea in our paper briefly.
Relevant literatures (e.g. \cite{Win10, Xi18, Xiang18}) reveal that the upper bounds on $n$ and $\nabla c$ for system \eqref{KSL} in a smooth bounded domain $\Omega$ depend on $|\Omega|$. This intuitively shows us that the methods that deal with bounded domain cases fail to directly apply in Cauchy problem. Roughly speaking, in a smooth bounded domain $\Omega$, from E.q. $\eqref{KSL}_1$ we can infer by  the damping effect of the logistic source that
\[\frac{\dd }{\dd t}\int_{\Omega}n \,\dd x+\int_{\Omega} n \,\dd x=\int_{\Omega}(\lambda+1)n-\mu n^2 \,\dd x\le \frac{(\lambda+1)^2|\Omega|}{4\mu},\]
where $n\ge 0$, moreover, we get the global-in-time upper bound
$$\|n(t)\|_{L^1(\Omega)}\le \max\left\{\|n_0\|_{L^1(\Omega)},\,\, \frac{(\lambda+1)^2|\Omega|}{4\mu}\right\}.$$
In the same way, one can obtain the global-in-time upper bound for $\|\nabla c\|_{L^2(\Omega)}$. Furthermore, both uniformly bounds can act as a starter for improving regularity to $L^p\times \dot W^{1,2p}$ by Moser-type iteration, within which for some large $p>1$ one enables to infer the uniform boundedness ($L^\infty$-norm) by standard $L^p$-$L^q$ estimates of heat semigroup.
Unfortunately, these   global-in-time upper bounds strongly depend on the measure of  domain $\Omega$. So it does not hold for the unbound domain such as the entire space $\RR^d.$ The main reason is that  logistic term can not provide the damping effect in $\RR^d$. Without damping effect, we only get the following  upper bounds which  depend on time $t$ (cf.  Proposition \ref{nL1cL2}), that is,
\[\|n(t)\|_{L^1(\RR^d)} \le e^{\lambda t}\|n_0\|_{L^1(\RR^d)}\,\, \text{ and }\,\,\|\nabla c(t)\|^2_{L^2(\RR^d)}\le\|\nabla c_0\|^2_{L^2(\RR^d)}+\frac{e^{\lambda t}}{\tau}\|n_0\|_{L^1(\RR^d)}.\]
According to this, we can show the $L^\infty(\RR^d)$-norm of $(n,\nabla c)$ enjoys upper bound dependent of $t$ by repeating the iteration procedure in \cite{Win10}. And then we get the global-in-time well-posedness for Cauchy problem via the standard argument. However the problem concerning  global-in-time boundedness of solution is not solved.  To do this, we have a new observation  that $L^\infty$-norm characterizes a ``local property'' to some extend, that is
$  L^\infty(\RR^d)\hookrightarrow L^p_{\uloc}(\RR^d)$ for each $1\le p<\infty$.
This property inspires us to perform  the same procedure in \cite{Win10} in the uniformly local space framework.  Of course, the implementation of the process is not effortless. It takes some new techniques to cope with new difficulties (compared to bounded domain cases) to preform this process. For example,  the space localization equation of $n$ takes
 \[(\phi n)_t-\Delta(\phi n)=-\nabla\cdot\big((\phi n)\nabla c\big)+\lambda(\phi n)-\mu \phi n^2-\Delta\phi n-2\nabla n\cdot\nabla\phi+n\nabla\phi\cdot\nabla c,\]
 where $\phi$ is a smooth cut-off function. Moreover, it is achieved  by integrating in space variable that
 \[\frac{\dd}{\dd t}\int_{\RR^d}\phi n\dd x= \lambda\int_{\RR^d}\phi n\dd x-\mu\int_{\RR^d}\phi n^2\dd x+\int_{\RR^d}\Delta\phi n\dd x+\int_{\RR^d}n\nabla\phi\cdot\nabla c\dd x.  \]
 Since lack of regularity of $\nabla c$, it seems impossible to establish the natural estimate of $ \|n\|_{L^1_{\uloc}(\RR^d)}$ independently, which is cornerstone of ``a priori" estimates in previous results. To accomplish  this task, by taking advantage of coupling structure, a new combination of $L^1_{\uloc}(\RR^d)$-norm of $n$ and $L^2_{\uloc}(\RR^d)$-norm of $\nabla c$ is applied, similar to \eqref{z}. As a result, it  reduces redundant items and close the estimates by choosing  a appropriate test function.

The rest of paper is organized as follows. Section 2 is devoted to proving the local-in-time solution. In section 3, we study the global-in-time well-posedness and boundedness of solution to Cauchy problem.

\noindent\textbf{Notation.} Throughout the paper, $\mathbb{R}^{+}=(0, \infty)$ and $C$ stands for a "generic" constant which may changes from line to line. For $p, q \in[1, \infty]$, the usual Lebesgue space is denoted by $L^{p}\left(\mathbb{R}^{d}\right)$ and $\|\cdot\|_{L_{t}^{q} L^{p}(\RR^d)}$ denotes the norm of $\displaystyle\left(\int_{0}^{t}\|\cdot\|_{L^p(\RR^d)}^{q} \,\textnormal{d} s\right)^{\frac{1}{q}}$. Let us denote
 $$B_{r}\left(x_{0}\right):=\left\{x \in \mathbb{R}^{d}:\left|x-x_{0}\right|<r\right\}.$$
For $1\leq p<\infty, $ we define
$$
\begin{aligned}
&\|f\|_{p, \lambda}:=\sup _{x \in \mathbb{R}^{d}}\left(\|f\|_{L^{p}\left(B_{\lambda}(x)\right)}\right)=\sup _{x \in \mathbb{R}^{d}}\left(\int_{|x-y|<\lambda}|f(y)|^{p} \mathrm{~d} y\right)^{1 / p}, \\
&L_{\uloc}^{p}\left(\mathbb{R}^{d}\right):=\left\{f \in L_{\mathrm{loc}}^{1}\left(\mathbb{R}^{d}\right) ;\|f\|_{p, 1}<\infty\right\}, \quad\|f\|_{L_{\uloc}^{p}\left(\mathbb{R}^{d}\right)}:=\|f\|_{p, 1} .
\end{aligned}
$$
 Finally, $\mathcal{D}\left(\mathbb{R}^{d}\right)$ is a space of smooth compactly supported functions on $\mathbb{R}^{d}$, and space $W^{s,p}\left(\mathbb{R}^{d}\right)$ is the general Sobolev spaces with $1\leq p\leq \infty$ and $s\geq0$.
\section{Local-in-time well-posedness for Cauchy problem}
In this section, we are going to show local-in-time existence and uniqueness of the classical solution  to system~\eqref{che}-\eqref{che-I}. Let us firstly give the statement concerning local-in-time well-posedness.
\begin{theorem} [Local-in-time well-posedness]\label{thm-local} Let $\chi,\,\tau,\, \mu>0$, $\lambda\ge 0, $ $d\ge 1$ and the initial data $(n_0, c_0)$ satisfy
 \begin{equation}\label{ini-I}
\left\{ \aligned
&n_0\in  L^1(\RR^d)\cap L^\infty(\RR^d)\,\text{ is nonnegative},\\
&c_0\in H^1(\RR^d)\cap W^{1,\infty}(\RR^d) \, \text{ is nonnegative}.
\endaligned
\right.
\end{equation}Then there exist a maximal $T_{\max}\in (0, \infty]$ and a uniquely nonnegative classical solution $(n, c)$ to  system~\eqref{che}-\eqref{che-I} satisfying
\begin{align*}
&n\in C^0( [0, T_{\max}); L^1(\RR^d)\cap L^\infty(\RR^d))\cap C^{2,1}(\RR^d\times (0, T_{\max})),\\
&c\in C^0([0, T_{\max});  H^1(\RR^d)\cap W^{1,\infty}(\RR^d))\cap C^{2,1}(\RR^d\times (0, T_{\max})).
\end{align*}
 Moreover, if $T_{\max}<\infty$, then
\begin{equation}\label{blowup1}
 \lim_{T\to T_{ \textnormal{max}}-}\left(\|n(\cdot, t)\|_{L^\infty(\RR^d)} +\|c(\cdot, t)\|_{W^{1,\infty}(\RR^d)}\right)=\infty.
\end{equation}
\end{theorem}
We will proceed to prove it step by step by using the standard argument as used in \cite{Win20}.
 \noindent \textbf{Step 1: Existence}.\,\, According to hypotheses \eqref{ini-I}, there exists a constant $M>0$ such that the initial data satisfy
 \[\|n_0\|_{L^1(\RR^d)}+\|n_0\|_{L^\infty(\RR^d)}\le M\]
 and \[\|c_0\|_{H^1(\RR^d)}+\|c_0\|_{W^{1,\infty}(\RR^d)}\le M.\]
For small $T\in (0, 1)$, we choose the following space $X$ as the solution space:
\[X_T:=C^0\big([0, T];\, \left(  L^1(\RR^d)\cap L^\infty(\RR^d))\times ( H^1(\RR^d)\cap W^{1,\infty}(\RR^d))\right),\]
 and the closed subset $S$ is defined by
\[
\begin{split}S:=\Big\{(n, c)\in X_T\Big|\,\, &\|n\|_{L^\infty\big([0, T]; L^1(\RR^d)\big)}+\|n\|_{L^\infty\big([0, T]; L^\infty(\RR^d)\big)}\le 2M,\\& \|c\|_{L^\infty\big([0, T]; H^1(\RR^d)\big)}+\|c\|_{L^\infty\big([0, T]; W^{1,\infty}(\RR^d)\big)}\le 2M\Big\}.
\end{split}\]
Next, we consider  a mapping $\Phi=(\Phi_1, \Phi_2)$ on $S$ defined by
\begin{align*}
\Phi_1(n,c)=&e^{t\Delta} n_0-\chi\int_0^t e^{(t-s)\Delta} \nabla\cdot(n\nabla c)\dd s+\lambda\int_0^t e^{(t-s)\Delta} n\dd s -\mu\int_0^t e^{(t-s)\Delta} n^2\dd s
\end{align*}
and
\begin{align*}
&\Phi_2(n,c)=e^{\frac{t}{\tau}(-1+\Delta)}c_0+\frac{1}{\tau}\int_0^te^{\frac{(t-s)}{\tau}(-1+\Delta)}n\dd s
\end{align*}
for each $t\in[0, T].$
It is easy to check that $t\mapsto e^{t\Delta} n_0$ belongs to $C^0([0, \infty); L^1(\RR^d)\cap L^\infty(\RR^d))$ and $t\mapsto e^{t(-1+\Delta)} c_0$ belongs to $C^0([0, \infty); H^1(\RR^d)\cap W^{1,\infty}(\RR^d))$. Collecting with some basic regularity properties concerning semigroup $e^{t\Delta}\,(e^{t(-1+\Delta)})$, for example established in \cite[Lemma 2.1]{Win20}, we can deduce that $\Phi$ maps $S$ into $X_T$.
Taking advantage of standard $L^p$-$L^q$ estimate for heat semigroup, we can show that  for all $t\in [0, T]$,
\begin{align*}
 \|\Phi_1(n,c)(t)\|_{L^1(\RR^d)}
\le&\|n_0\|_{L^1(\RR^d)}+C_1\chi\int_0^t(t-s)^{-\frac{1}{2}}\|n(s)\|_{L^2(\RR^d)}\|\nabla c(s)\|_{L^2(\RR^d)}\dd s
\\&+\lambda\int_0^t\|n(s)\|_{L^1(\RR^d)}\dd s+\mu \int_0^t\|n(s)\|^2_{L^2(\RR^d)}\dd s\\
\le &\|n_0\|_{L^1(\RR^d)}+C_1\chi T^{\frac{1}{2}}\|n\|^{\frac{1}{2}}_{L^\infty_T( L^1(\RR^d))}\|n\|^{\frac{1}{2}}_{L^\infty_T(L^\infty(\RR^d))}\|\nabla c\|_{L^\infty_T(L^2(\RR^d))}
\\&+\lambda T\|n\|_{L^\infty_T( L^1(\RR^d))}+\mu T\|n\|_{L^\infty_T( L^1(\RR^d))}\|n\|_{L^\infty_T(L^\infty(\RR^d))}.
\end{align*}
Similarly, by H\"{o}lder's and Young's inequalities, it yields that for $0\le t\le T$,
\begin{align*}
 \|\Phi_1(n,c)(t)\|_{L^\infty(\RR^d)}
\le & \|n_0\|_{L^\infty(\RR^d)}+C_2\chi\int_0^t(t-s)^{-\frac{1}{2}}\|n(s)\|_{L^\infty(\RR^d)}\|\nabla c(s)\|_{L^\infty(\RR^d)}\dd s\\&+\lambda\int_0^t\|n(s)\|_{L^\infty(\RR^d)}\dd s+\mu \int_0^t\|n(s)\|^2_{L^\infty(\RR^d)}\dd s\\
\le &\|n_0\|_{L^\infty(\RR^d)}+C_2\chi T^{\frac{1}{2}}\|n\|_{L^\infty_T( L^\infty(\RR^d))}\|\nabla c\|_{L^\infty_T( L^\infty(\RR^d))}\\&+\lambda T\|n\|_{L^\infty_T( L^\infty(\RR^d))}+\mu T\|n\|^2_{L^\infty_T(L^\infty(\RR^d))}.
\end{align*}
Therefore, adding the above two inequality leads to
\begin{align*}
&\|\Phi_1(n,c)(t)\|_{L^1(\RR^d)}+\|\Phi_1(n,c)(t)\|_{L^\infty(\RR^d)}\\
\le& \|n_0\|_{L^1(\RR^d)}+\|n_0\|_{L^\infty(\RR^d)}
+C_1\chi T^{\frac{1}{2}}\left(\|n\|_{L^\infty_T( L^1(\RR^d))}+\|n\|_{L^\infty_T(L^\infty(\RR^d))}\right)\|\nabla c\|_{L^\infty_T( L^2(\RR^d))}\\
&+C_2\chi T^{\frac{1}{2}}\|n\|_{L^\infty_T( L^\infty(\RR^d))}\|\nabla c\|_{L^\infty_T(L^\infty(\RR^d))}
+\lambda T\left(\|n\|_{L^\infty_T( L^1(\RR^d))}+\|n\|_{L^\infty_T(L^\infty(\RR^d))}\right)\\
&+\mu T\|n\|_{L^\infty_T( L^\infty(\RR^d))}\left(\|n\|_{L^\infty_T( L^1(\RR^d))}+\|n\|_{L^\infty_T( L^\infty(\RR^d))}\right)\\
\le &M+4C_1\chi T^{\frac{1}{2}}M^2+4C_2\chi T^{\frac{1}{2}}M^2+2\lambda TM+4\mu T M^2.
\end{align*}
If we choose
$$T_1=\min\left\{1,\, \frac{1}{4\lambda},\, \big(8M(C_1\chi+C_2\chi+\mu)\big)^{-\frac{1}{2}}\right\}, $$
it is easy to check that for $0<t\le T_1$,
$$\|\Phi_1(n,c)(t)\|_{L^1(\RR^d)}+\|\Phi_1(n,c)(t)\|_{L^\infty(\RR^d)}\le 2M.$$
In the similar fashion as above, we can show
\begin{align*}
&\|\Phi_2(n,c)(t)\|_{H^1(\RR^d)}+\|\Phi_2(n,c)(t)\|_{W^{1,\infty}(\RR^d)}\\
\le& e^{-\frac{t}{\tau}}\left(\|c_0\|_{H^1(\RR^d)}+\|c_0\|_{W^{1,\infty}(\RR^d)}\right)+\frac{1}{\tau}\int_0^t e^{-\frac{t-s}{\tau}}\left(\|n(s)\|_{L^2(\RR^d)}+\|n(s)\|_{L^\infty(\RR^d)}\right)\dd s\\
&+\frac{C_1}{{{\tau}}}\int_0^t \left(\frac{t-s}{\tau}\right)^{-\frac{1}{2}} e^{-\frac{t-s}{\tau}}\left(\|n(s)\|_{L^2(\RR^d)}+\|n(s)\|_{L^\infty(\RR^d)}\right)\dd s\\
\le&\left(\|c_0\|_{H^1(\RR^d)}+\|c_0\|_{W^{1,\infty}(\RR^d)}\right)+2\left(\frac{T}{\tau}
+C_1\left(\frac{T}{\tau}\right)^{\frac{1}{2}}\right)\left(\|n\|_{L^\infty_T( L^1(\RR^d))}+\|n\|_{L^\infty_T( L^\infty(\RR^d))}\right)\\
\le&M+4M\bigg(\frac{T}{\tau}+C_1\left(\frac{T}{\tau}\right)^{\frac{1}{2}}\bigg).
\end{align*}
Taking $$T_2=\min\left\{1, \,\left(\frac{\tau}{4(1+C_1\sqrt{\tau})}\right)^2\right\},$$ we obtain that $$\|\Phi_2(n,c)(t)\|_{H^1(\RR^d)}+\|\Phi_2(n,c)(t)\|_{W^{1,\infty}(\RR^d)}\le 2M.$$
Then $\Phi(t)$ maps $S$ into itself for $t\le T_0:=\min\{T_1, T_2\}$.
Next, we show $\Phi$ is a contraction mapping in a short time. For any couple $(n_1, c_1)\in S$ and couple $(n_2, c_2)\in S$, an easy computation yields that
\begin{align*}
&\|\Phi_1(n_1, c_1)(t)-\Phi_1(n_2, c_2)(t)\|_{L^1(\RR^d)}\\
\le& \chi\int_0^t\left\|e^{(t-s)\Delta}\nabla \cdot\left((n_1(s)\nabla c_1(s))-(n_2(s)\nabla c_2(s))\right)\right\|_{L^1(\RR^d)}\dd s\\
&+\lambda\int_0^t\left\|e^{(t-s)\Delta}\left(n_1(s)-n_2(s)\right)\right\|_{L^1(\RR^d)}\dd s+\mu\int_0^t\left\|e^{(t-s)\Delta}\left(n^2_1(s)-n^2_2(s)\right)\right\|_{L^1(\RR^d)}\dd s\\
\le &\chi\int_0^t\left\|e^{(t-s)\Delta}\nabla \cdot((n_1(s)-n_2(s))\nabla c_1(s))\right\|_{L^1(\RR^d)}\dd s\\&+\chi\int_0^t\left\|e^{(t-s)\Delta}\nabla \cdot(n_2(s)\nabla (c_1(s)-c_2(s))\right\|_{L^1(\RR^d)}\dd s\\
&+\lambda T\|n_1-n_2\|_{L^\infty_T(L^1(\RR^d))}+\mu T\|n_1+n_2\|_{L^\infty_T( L^\infty(\RR^d))}\|n_1-n_2\|_{L^\infty_T( L^1(\RR^d))}\\
\le& C_1\chi T^{\frac{1}{2}}\|n_1-n_2\|_{L^\infty_T( L^2(\RR^d))}\| c_2\|_{L^\infty_T(\dot{H}^1(\RR^d))}
+C_1\chi T^{\frac{1}{2}}\|n_2\|_{L^\infty_T( L^2(\RR^d))}\|c_1-c_2\|_{L^\infty_T(\dot{H}^1(\RR^d))}\\
&+\lambda T\|n_1-n_2\|_{L^\infty_T( L^1(\RR^d))}+\mu T\left(\|n_1\|_{L^\infty_T( L^\infty(\RR^d))}+\|n_2\|_{L^\infty_T( L^\infty(\RR^d))}\right)\|n_1-n_2\|_{L^\infty_T(L^1(\RR^d))}\\
\le&\left(4C_1M\chi T^{\frac{1}{2}}+\lambda T+4M\mu T\right)\|(n_1, c_1)-(n_2,c_2)\|_{X_T}.
\end{align*}
Similarly, we estimate
\begin{align*}
&\|\Phi_1(n_1, c_1)(t)-\Phi_1(n_2, c_2)(t)\|_{L^\infty(\RR^d)}\\
\le& C_4\chi T^{\frac{1}{2}-\frac{d}{2q}}\|n_1-n_2\|_{L^\infty_T( L^\infty(\RR^d))}\|\nabla c_2\|_{L^\infty_T( L^\infty(\RR^d))}\\&
+C_4\chi T^{\frac{1}{2}-\frac{d}{2q}}\|n_2\|_{L^\infty_T( L^\infty(\RR^d))}\|\nabla (c_1-c_2)\|_{L^\infty_T( L^\infty(\RR^d))}+\lambda T\|n_1-n_2\|_{L^\infty_T( L^\infty(\RR^d))}\\&+\mu T\left(\|n_1\|_{L^\infty_T( L^\infty(\RR^d))}+\|n_2\|_{L^\infty_T(L^\infty(\RR^d))}\right)\|n_1-n_2\|_{L^\infty_T( L^\infty(\RR^d))}\\
\le&\left(4C_4M\chi T^{\frac{1}{2}}+\lambda T+4M\mu T\right)\|(n_1, c_1)-(n_2,c_2)\|_{X_T},
\end{align*}
and
\begin{align*}
&\|\Phi_2(n_1, c_1)(t)-\Phi_2(n_2, c_2)(t)\|_{H^1(\RR^d)}+\|\Phi_2(n_1, c_1)(t)-\Phi_2(n_2, c_2)(t)\|_{W^{1,\infty}(\RR^d)}\\
\le&\frac{C_1}{\tau}\int_0^t \bigg(1+\left(\frac{t-s}{\tau}\right)^{-\frac{1}{2}}\bigg)\left(\|n_1(s)-n_2(s)\|_{L^2(\RR^d)}
+\|n_1(s)-n_2(s)\|_{L^\infty(\RR^d)}\right)\dd s\\
\le& C_5\left(\frac{ T}{\tau}+\left(\frac{T}{\tau}\right)^{\frac{1}{2}}\right)\|n_1-n_2\|_{X_T}.
\end{align*}
Taking
$$\widetilde{T}=\min\left\{T_0, \,{\left(16MC_4\chi+4\lambda+16M\mu\right)^{\frac{1}{2}-\frac{d}{2q}}}, \, \frac{\tau^2}{4C_5^2(1+\sqrt{\tau})^2}\right\},$$
then $\Phi$ is a contractive mapping on $S$. We apply the Banach fixed point theorem to obtain  that there exists $(n ,c)\in S$ such that $\Phi(n, c)=(n, c)$. Moreover, by standard arguments involving semigroup estimates (cf. Lemma 3.3 in \cite{FIWY16}), it can be checked that $(n, c)$ belongs to $C^{2,1}(\RR^d\times (0, \widetilde{T}))^2$ and solves problem \eqref{che}-\eqref{che-I} classically in $(0, \widetilde{T})\times\RR^d$.\\
\textbf{Step 2: Nonnegativity}\,\, We are going to prove that $n, c\ge0$ if initial data $n_0, c_0\ge 0$.
We firstly choose is a smooth cut-off function $\phi$ satisfying
\begin{equation}\label{eq-1-cutoff}
\phi(x)=\begin{cases}
1,\quad x\in B_R(0),\\
0,\quad x\in \mathbb{R}^d\backslash B_{2R}(0),
\end{cases}
\end{equation}
where $R>0$ is a real number.

According to the local property $n\in C^{2,1}(\RR^d\times (0, \widetilde{T})),$ the global  property $$\partial_t(\phi n),\,\,\,\phi\nabla n\in L^2(\RR^d)  $$ can be achieved by introducing the cut-off function $\phi$.
Let us denote  $$n^-=\max\{0, -n\}. $$
Multiplying E.q.$\eqref{che}_1$ by $\phi^2n^-$ and integrating, it yields that
\begin{align*}
&\frac{1}{2}\frac{\dd}{\dd t}\|\phi n^-(t)\|^2_{L^2(\RR^d)}+\|\phi\nabla n^-(t)\|^2_{L^2(\RR^d)}\\
\le& \|\nabla c\|_{L^\infty(\RR^d)}\|\phi\nabla n^-\|_{L^2(\RR^d)}\|\phi n^-\|_{L^{2}(\RR^d)}+\|n\|^2_{L^2(\RR^3)}\|\nabla\phi\|_{L^\infty(\RR^d)}\|\nabla c\|_{L^\infty(\RR^d)}\\
&+2\|n\|_{L^2(\RR^d)}\|\nabla\phi\|_{L^\infty(\RR^d)}\|\phi\nabla n^-\|_{L^2(\RR^d)}+\lambda\|\phi n^-\|^2_{L^2(\RR^d)}\\
\le&2 \|\nabla c\|^{2}_{L^\infty(\RR^d)}\|\phi n^-\|^{2}_{L^{2}(\RR^d)}+\frac{C}{R}\|n\|^2_{L^2(\RR^d)} \|\nabla c\|_{L^\infty(\RR^3)}+\frac{C}{R}\|n\|^2_{L^2(\RR^d)}+\frac{1}{2}\|\phi\nabla n^-\|^{2}_{L^2(\RR^d)}\\&+\lambda\|\phi n^-\|^2_{L^2(\RR^d)}.
\end{align*}
Since $(n,c)\in X_{\widetilde{T}}$, we obtain by  Gr\"{o}nwall's inequality that  for any $t<\widetilde{T}$,
\begin{align*}
\|\phi n^-(t)\|^2_{L^2(\RR^d)}\le \exp\left(\int_0^t2\|\nabla c(s)\|^{2}_{L^\infty(\RR^d)}\dd s+\lambda t\right)\|n^-_0\|^2_{L^2(\RR^d)} +\frac{C}{R}\widetilde{T}.
\end{align*}
Taking $R$ goes to infinite, we immediately get by the continuous property that $n(x, t)\ge 0$ in the domain $\RR^d\times(0,\widetilde{T})$.

Letting $\tilde{c}=e^tc,$ it follows from $\eqref{che}_2$ and the nonnegativity of $n$ that
\[\partial_t\tilde{c}-\Delta\tilde{c}=e^tn\geq0,\]
which means that $\tilde{c}$ is the super-solution of the linear heat equation. Since $\tilde{c}\in H^1(\RR^d)\cap W^{1,\infty}(\RR^d)),$ we get  by weak maximum principle for Cauchy problem that
\[\inf_{x\in\RR^d\times[0,\widetilde{T})}\tilde{c}(x,t)=\inf_{x\in\RR^d}\tilde{c}(x,0)=\inf_{x\in\RR^d} {c}_0(x).\]
Hence
\[c(x,t)=e^{-t}\tilde{c}(x,t)\geq e^{-t}\inf_{x\in\RR^d} {c}_0(x)\geq0\quad\text{for all}\quad(x,t)\in\RR^d\times\big(0,\widetilde{T}\big),\]
because $c_0\ge 0$.\\
\noindent\textbf{Step 3: Uniqueness}\,\,First of all, we denote the maximal lifespan by $T_{\max}$. Now we begin to prove the uniqueness of the solution of \eqref{che}.  Since
 $$(n,c)\in C^0\big([0, T_{\max});(L^1(\RR^d)\cap L^\infty(\RR^d))\times  H^1\cap W^{1,\infty}(\RR^d)\big) ,$$ we see that the couple $(n,c)$ vanishes to $0$ as $|x|\to \infty$. Therefore, the following estimates have no boundary term involving integrating by parts.
Let $(n , c)$ and $(\widetilde{n}, \widetilde{c})$ be the two solutions to system \eqref{che} in $\RR^d\times (0, T_{\max})$ with the same initial data. Set $\delta n=n-\widetilde{n}$ and $\delta c=c-\widetilde{c}$, we easily find that the couple $(\delta n,\delta c)$ satisfies
 \begin{equation}\label{KSL-diff}
\left. \aligned
  \partial_t\delta n-\Delta \delta n=& -\chi\nabla\cdot(\delta n\nabla c)-\chi\nabla\cdot(\widetilde{n}\nabla \delta c)+\lambda\delta n-\mu (n+\widetilde{n})\delta n\\
 \tau\partial_t\delta c-\Delta \delta c=&-\delta c+\delta n
\endaligned
\right\} \text{ in } \RR^d\times(0,T_{\max}),
\end{equation}
which is supplemented with the homogeneous initial data
\[(\delta n,\delta c)(x,0)=(0,0)\quad\text{in}\quad\RR^d.\]
Then we can show by using energy method and introducing the cut-off function $\phi$ defined in \eqref{eq-1-cutoff} that
\begin{align*}
&\frac{1}{2}\frac{\dd}{\dd t}\|\phi\delta n(t)\|^2_{L^2(\RR^d)}+\|\phi\nabla \delta n(t)\|^2_{L^2(\RR^d)}\\
\le&\chi\int_{\RR^d}\phi^2\delta n\nabla c\cdot\nabla\delta n\dd x+2\chi\int_{\RR^d}\phi\delta n\nabla c\cdot\nabla\phi\delta n\dd x+\chi\int_{\RR^d} \phi^2\widetilde{n}\nabla\delta c\cdot\nabla\delta n\dd x\\&+2\chi\int_{\RR^d} \phi \widetilde{n}\nabla\delta c\cdot\nabla\phi\delta n\dd x-2\int_{\RR^d}\phi\nabla\delta n\cdot\nabla\phi\delta n\mathrm{d}x\\
&+\|\phi\delta n\|^2_{L^2}(\lambda+\mu\|n\|_{L^\infty(\RR^d)}+\mu\|\widetilde{n}\|_{L^\infty(\RR^d)})\\
\le &{\chi}\|\nabla c\|_{L^\infty(\RR^d)}\|\phi\delta n\|_{L^{2}(\RR^d)}\|\phi\nabla\delta n\|_{L^2}+\chi\|\widetilde{n}\|_{L^\infty(\RR^d)}\|\phi\nabla\delta c\|_{L^2(\RR^d)}\|\phi\nabla \delta n\|_{L^2(\RR^d)}\\&+2\chi\|\nabla \phi\|_{L^\infty(\RR^d)}\|\delta n\|_{L^2(\RR^d)}\left(\|\nabla c\|_{L^\infty(\RR^d)}\|\phi\delta n\|_{L^2(\RR^d)}
 +\|\widetilde{n}\|_{L^\infty(\RR^d)}\|\phi\nabla\delta c\|_{L^2(\RR^d)}\right)
\\&+\|\nabla \phi\|_{L^\infty(\RR^d)}\|\phi\nabla\delta n\|_{L^2(\RR^d)}\|\delta n\|_{L^2(\RR^d)} +\|\phi\delta n\|^2_{L^2(\RR^d)}\left(\lambda+\mu\|n\|_{L^\infty(\RR^d)}+\mu\|\widetilde{n}\|_{L^\infty(\RR^d)}\right)\\
\le&\frac{1}{2}\|\phi\nabla \delta n\|^2_{L^2(\RR^d)}+C\chi^{2}\|\nabla c\|^{2}_{L^\infty(\RR^d)}\|\phi\delta n\|^{2}_{L^{2}(\RR^d)}+C\chi^2\|\widetilde{n}\|^2_{L^\infty(\RR^d)}\|\phi\nabla\delta c\|^2_{L^2(\RR^d)}\\&+\|\phi\delta n\|^2_{L^2(\RR^d)}\left(\lambda+\mu\|n\|_{L^\infty(\RR^d)}+\mu\|\widetilde{n}\|_{L^\infty(\RR^d)}\right)+\frac{C}{R},
\end{align*}
in the last line, we have used the fact
\[(n,c),\,(\widetilde{n}, \widetilde{c})\in C^0\big([0, T_{\max});(L^1(\RR^d)\cap L^\infty(\RR^d))\times  H^1\cap W^{1,\infty}(\RR^d)\big).\]
Similarly, with aid of H\"{o}lder's and Young's inequalities, one gets that
\begin{align*}
&\frac{\tau}{2}\frac{\dd }{\dd t}\|\phi\nabla\delta c(t)\|^2_{L^2(\RR^d)}+\|\phi\Delta\delta c(t)\|^2_{L^2(\RR^d)}+\|\phi\nabla\delta c(t)\|^2_{L^2(\RR^d)}\\=&\int_{\RR^d}\phi^2\nabla \delta n\cdot\nabla \delta c\dd x-2\int_{\RR^d}\phi\Delta \delta c\nabla\phi\cdot\nabla \delta c\dd x\\
\le& \|\phi\delta n\|^2_{L^2(\RR^d)}+\frac{1}{2}\|\phi\Delta\delta c\|^2_{L^2(\RR^d)}+8\|\nabla\phi\|^2_{L^\infty(\RR^d)}\| \nabla\delta c\|^2_{L^2(\RR^d)}.
\end{align*}
Therefore, combined the above two estimates shows that
\begin{align*}
&\frac{\dd}{\dd t}\left(\frac{1}{2}\|\phi\delta n(t)\|^2_{L^2(\RR^d)}+\frac{\tau}{2}\|\phi\nabla\delta c(t)\|^2_{L^2(\RR^d)}\right)\\
\le& \|\phi\delta n\|^{2}_{L^{2}(\RR^d)}\left(C\chi^{2}\|\nabla c\|^{2}_{L^\infty(\RR^d)}+\lambda+\mu\left(\|n\|_{L^\infty(\RR^d)}+\|\widetilde{n}\|_{L^\infty(\RR^d)}\right)+1\right) 
\\
&+C\chi^2\|\widetilde{n}\|^2_{L^\infty(\RR^d)}\|\phi\nabla\delta c\|^2_{L^2(\RR^d)}+\frac{C}{R}+\frac{C}{R^2}.
\end{align*}
Integrating the above inequality with respect to $t$ and then taking $R$ goes to infinite yields 
\begin{align*}
& \frac{1}{2}\|\delta n(t)\|^2_{L^2(\RR^d)}+\frac{\tau}{2}\|\nabla\delta c(t)\|^2_{L^2(\RR^d)} \\
\le& \int_0^t\|\delta n\|^{2}_{L^{2}(\RR^d)}\left(C\chi^{2}\|\nabla c\|^{2}_{L^\infty(\RR^d)}+\lambda+\mu\left(\|n\|_{L^\infty(\RR^d)}+\|\widetilde{n}\|_{L^\infty(\RR^d)}\right)+1\right) \dd x
\\
&+C\chi^2\int_0^t\|\widetilde{n}\|^2_{L^\infty(\RR^d)}\|\nabla\delta c\|^2_{L^2(\RR^d)}\dd s.
\end{align*}
Noting the fact that $n,\, \widetilde{n}\in L^\infty_{\rm loc}([0, T_{\max}),\, L^\infty(\RR^d))$ and $c,\, \widetilde{c}\in L^\infty_{\rm loc}([0, T_{\max}),\, W^{1, \infty}(\RR^d))$, by Gr\"{o}nwall inequality we obtain that $(n, c)\equiv (\widetilde{n}, \widetilde{c})$ in $\RR^d\times (0, T_{\max})$.

From  our definition of $T$ it becomes clear through another standard reasoning that $(n, c)$ can be extended up to a maximal $T_{\max}\in (0,\infty]$ fulfilling
\begin{equation}\label{blowup-2}
\|n(\cdot, t)\|_{L^1(\RR^d)}+\|n(\cdot, t)\|_{L^\infty(\RR^d)}+\|c(\cdot, t)\|_{H^1(\RR^d)}+\|c(\cdot, t)\|_{W^{1,\infty}(\RR^d)}\to \infty \,\,\,\,\text{as}\,\, t\nearrow T_{\max}.
\end{equation}
Now we adopt proof by contradiction to verify conclusion \eqref{blowup-2}. Suppose \eqref{blowup-2} does not hold, there exists a constant $R>0$ such that
\[\|n(\cdot, t)\|_{L^1(\RR^d)}+\|n(\cdot, t)\|_{L^\infty(\RR^d)}+\|c(\cdot, t)\|_{H^1(\RR^d)}+\|c(\cdot, t)\|_{W^{1,q}(\RR^d)}\le R.\]
According to selection method of $\widetilde{T}$, we choose
\begin{align*}
\varepsilon=\min\bigg\{1,\, \frac{1}{4\lambda}, \, &\left(8R(C_1\chi+C_2\chi+\mu)\right)^{2},\,\left(\frac{\tau}{4(1+C_1\sqrt{\tau})}\right)^2, \\ &\qquad\qquad\qquad\qquad\left(16RC_4\chi+4\lambda+16M\mu\right)^{2}, \, \frac{\tau^2}{4C_5^2(1+\sqrt{\tau})^2}\bigg\}.
\end{align*}
From the above existence proof, we conclude that there exists a unique solution $(\bar{n}, \bar{c})$ with initial data $\big(n(T_{\max}-{\varepsilon}/{2}), c(T_{\max}-{\varepsilon}/{2})\big)$ on $[0, \varepsilon]$, moreover, we have by uniqueness that
\[\bar{n}(t)=n\left(t+T_{\max}-{\varepsilon}/{2}\right)\quad\text{and}\quad \bar{c}(t)=c\left(t +T_{\max}-{\varepsilon}/{2}\right)\] on $[0, {\varepsilon}/{2})$, which implies that $(\bar{n}, \bar{c})$ extends solution $(n, c)$ beyond $T_{\max}$.  This contradicts the definition of $T_{\max}$. Hence \eqref{blowup-2} is valid. \\
\noindent\textbf{Step 4: Continuation}. Our task is now to show \eqref{blowup1}. Thanks to \eqref{blowup-2}, it suffice to prove that two quantities $\|n(t)\|_{L^1(\RR^d)}$ and $\|c(t)\|_{H^1(\RR^d)}$ cannot blow up at finite time.
\begin{proposition}\label{nL1cL2}Let $(n, c)$ be the solution of system \eqref{che}-\eqref{che-I}. Then, for all $t\in (0, T_{\max})$, the couple $(n,c)$ satisfies the following estimates independent of $t$:
\begin{equation}\label{nL1}
\|n(t)\|_{L^1(\RR^d)}+\int_0^t\|n(s)\|^2_{L^2(\RR^d)}\dd s\le e^{\lambda t}\|n_0\|_{L^1(\RR^d)},
\end{equation}
\begin{equation}\label{cL2}
\begin{split}
 \tau\|c(t)\|^2_{L^2(\RR^d)}+\int_0^t\|c(s)\|^2_{H^1(\RR^d)}\dd s 
 \le  \tau\|c_0\|^2_{L^2(\RR^d)}+e^{\lambda t}\|n_0\|_{L^1(\RR^d)}
 \end{split}
\end{equation}
and
\begin{equation}\label{ncL2}
\begin{split}
 \tau\|\nabla c(t)\|^2_{L^2(\RR^d)} +\int_0^t \|\nabla c(s)\|^2_{H^1(\RR^d)}\dd s \le \tau\|\nabla c_0\|^2_{L^2(\RR^d)}+e^{\lambda t}\|n_0\|_{L^1(\RR^d)}.
\end{split}
\end{equation}
\end{proposition}
\begin{proof}[Proof of Proposition \ref{nL1cL2}]
Owning to $n\ge 0$,  multiplying E.q.$\eqref{che}_1$ by $\phi$  and integrating with respect to spatial variable yields that
\begin{align*}
&\frac{\dd}{\dd t}\|\phi n(t)\|_{L^1(\RR^d)}+\mu\big\|\sqrt{\phi}n(t)\big\|^2_{L^2(\RR^d)}\\ \leq&\lambda\|\phi n(t)\|_{L^1(\RR^d)}+\frac{C}{R^2}\|n\|_{L^1(\RR^d)}+\frac{C}{R}\|\nabla c\|_{L^\infty(\RR^d)}\|n\|_{L^1(\RR^d)}.
\end{align*}
Performing Gr\"{o}nwall's inequality and taking $R$ goes to infinite, the above equality implies estimate \eqref{nL1}.

Multiplying E.q.$\eqref{che}_2$ by $\phi^2 c$ and integrating in terms of spatial variable, we obtain that
\begin{align*}
\frac{\tau}{2}\frac{\dd}{\dd t}\|\phi c(t)\|^2_{L^2(\RR^d)}+\frac{1}{2}\|\phi \nabla c(t)\|^2_{L^2(\RR^d)}+\frac{1}{2}\|\phi c(t)\|^2_{L^2(\RR^d)}\le \frac{1}{2}\|\phi n\|^2_{L^2(\RR^d)}+\frac{C}{R^2}\|c\|^2_{L^2(\RR^d)}.
\end{align*}
Integrating the above inequality with respect to time variable from $0$ to $t$ and taking $R$  to infinite lead to
\begin{align*}
 &\tau \| c(t)\|^2_{L^2(\RR^d)}+ \int_0^t\left(\| \nabla c(s)\|^2_{L^2(\RR^d)}+ \| c(s)\|^2_{L^2(\RR^d)}\right)\dd s \le \tau \|  c_0\|^2_{L^2(\RR^d)}+\int_0^t\| n(s)\|^2_{L^2(\RR^d)}\dd s,
\end{align*}
combining with estimate~\eqref{nL1} yields \eqref{cL2}.

Similarly, we can show 
\begin{align*}
&\frac{\tau}{2}\frac{\dd}{\dd t}\|\phi\nabla c(t)\|^2_{L^2(\RR^d)}+\|\phi\nabla^2 c(t)\|^2_{L^2(\RR^d)}+\|\phi\nabla c(t)\|^2_{L^2(\RR^d)}\\=&-\int_{\RR^d}  n\Delta c\dd x-2\int_{\RR^d}\phi \nabla c\cdot\big(\nabla\phi\cdot\nabla\nabla c\big)\dd x-2\int_{\RR^d}\phi n\nabla \phi\cdot\nabla c\dd x\\
\le& 2\|\phi n\|^2_{L^2(\RR^d)}+\frac{1}{2}\|\phi\nabla^2 c\|^2_{L^2(\RR^d)}+\frac{C}{R^2}\|\nabla c\|^2_{L^2(\RR^d)}.
\end{align*}
Thanks to inequality \eqref{nL1}, it follows from the above inequality that \eqref{n1c2}.
\end{proof}
\section{Proof of main result }
This section is devoted to show Theorem \ref{result}. We will adopt the following approximate schemes for Cauchy problem \eqref{che}-\eqref{che-I} to do this:
\begin{equation}\label{che-M}
\left. \aligned
 \partial_tn_M-\Delta n_M=&-\chi\nabla\cdot(n_M\nabla c_M)+\la n_M-\mu n_M^2 \\
 \tau\partial_tc_M-\Delta c_M=&-c_M+n_M
\endaligned
\right\}\quad\text{in}\,\,\,\,\RR^d\times\RR^+,
\end{equation}
\begin{equation}\label{che-I-M}
n_M(0, x)=\psi(x/M)n_0(x),\quad c_M(0, x)=\psi(x/M)c_0(x)\quad\text{in}\,\,\,\,\RR^d.
\end{equation}
Here $\psi$ is a smooth function satisfying
\[\psi(x)=\begin{cases}
1,\quad&x\in B_1(0)\\
0,&x\in \RR^d\backslash B_2(0)
\end{cases}.\]
Since $n_0\in L^\infty(\RR^d)$ and $c_0\in W^{1,\infty}(\RR^d)$,  it is easy to check that the initial data in \eqref{che-I-M} satisfy
\[n_M(0, x)\in L^1(\RR^d)\cap L^\infty(\RR^d)\]
and
\[c_M(0, x)\in H^1(\RR^d)\cap W^{1,\infty}(\RR^d).\]
Moreover, we have from Theorem \ref{thm-local} that problem \eqref{che-M}-\eqref{che-I-M} admits a unique nonnegative solution $(n_M,c_M)\in\big(C^{2,1}(\RR^d\times (0, T_{\max}))\big)^2$ such that
\begin{align*}
&n^M\in C^0\left( [0, T_{\max}); L^1(\RR^d)\cap L^\infty(\RR^d))\cap C^{2,1}(\RR^d\times (0, T_{\max})\right),\\
&c^M\in C^0\left([0, T_{\max});  H^1(\RR^d)\cap W^{1,\infty}(\RR^d))\cap C^{2,1}(\RR^d\times (0, T_{\max})\right),
\end{align*}
 if $T_{\max}<\infty$, then
\begin{equation}\label{blowup}
 \lim_{T\to T_{ \textnormal{max}}-}\left(\left\|n^M(\cdot, t)\right\|_{L^\infty(\RR^d)} +\left\|c^M(\cdot, t)\right\|_{W^{1,\infty}(\RR^d)}\right)=\infty.
\end{equation}
This implies the solutions to the regularized problem \eqref{che-M}-\eqref{che-I-M} are smooth and
decay sufficiently fast at infinity, so there are no boundary terms when we integrate
by parts in our calculations below. In addition, we define uniformly local space $L^{p,R}_{\uloc}(\RR^d)$ consists of the local integral function $f$ satisfying
\[\|f\|_{p,R}<+\infty,\]
 where $R$ is undermined now and fixed later. Let us point out that $L^{p,R}_{\uloc}(\RR^d)$ coincides with  space $L^{p }_{\uloc}(\RR^d)$ defined in Notation for arbitrary $R>0$. Indeed, it is obvious that for $R\geq1,$
 \[\|f\|_{p,1}\leq \|f\|_{p,R}.\]
 On the other hand, we have by the covering theorem that for $R\geq1,$
 \[ \|f\|_{p,R}\leq \left(CR^d\right)^{\frac{1}{p}}\|f\|_{p,1}.\]
 To simplify the notation, we agree that $(n,c):=(n_M,c_M)$ and $L^{p}_{\uloc}(\RR^d):=L^{p,R}_{\uloc}(\RR^d)$ in the following part of this section.
 Next we are going to establish  some useful \emph{a priori estimates} for $(n_M,c_M)$ by introducing the uniformly local space $L^{p,R}_{\uloc}(\RR^d)$.
  \begin{proposition}\label{n1c2}
Assume $\tau, \mu, \chi>0$ and $\lambda\ge 0$. Let the couple $(n, c)$ be the solution of system \eqref{che-M}-\eqref{che-I-M}, if $\mu>\frac{d\chi}{4}$, then there exists a $R_0>1$ such that for all $R\ge R_0$ and $T\in(0, T_{\textnormal{max}})$,
\begin{align*}
&\|n\|_{L^\infty_TL^1_{\uloc}(\RR^d)}+\frac{\chi\tau}{4}\|\nabla c\|^2_{L^\infty_TL^2_{\uloc}(\RR^d)}\\
\le &4\|n_0\|_{L^1_{\uloc}(\RR^d)}+2{\chi\tau}\|\nabla c_0\|^2_{L^2_{\uloc}(\RR^d)}+C(\lambda, \tau, \mu, \chi, d, R).
\end{align*}
\end{proposition}
\begin{proof}
We introduce a function $\phi^R_{x_0}\in C^\infty_c(\RR^d)$ by
\begin{equation}
\phi^{R}_{x_0}(x)=\left\{ \aligned&\exp\left(\frac{4}{3}+\frac{4R^2}{|x-x_0|^2-4R^2}\right),&|x-x_0|< 2R,\\
&0,&|x-x_0|\ge 2R.
\endaligned
\right.
\end{equation}
From the definition of $\phi^R_{x_0}$, it is easy to check that
\begin{itemize}
  \item [(i)] $\displaystyle1\le\phi^R_{x_0}(x)< 2$, \quad $\forall\,x\in B_{R}(x_0)$.\smallskip
  \item [(ii)]$\displaystyle \phi^R_{x_0}\big|_{\partial B_{2R}(x_0)}=\frac{\partial\phi^R_{x_0}}{\partial \nu}\bigg|_{\partial B_{2R}(x_0)}=0$.\smallskip
   \item [(iii)] $\displaystyle |\nabla \phi^R_{x_0}|\le \frac{C}{R}$, \,$\displaystyle|D^2 \phi^R_{x_0}|\le \frac{C}{R^2}$, where $C$ is an absolute constant independent of $R$.
\end{itemize}
To begin with, using $$\Delta\nabla c\cdot\nabla c=\frac{1}{2}\Delta|\nabla c|^2-|D^2 c|^2, $$ we can infer from system \eqref{che} that
\begin{align*}
\frac{\dd}{\dd t}\left(n+\frac{\chi\tau}{2}|\nabla c|^2\right)-\Delta\left(n+\frac{\chi}{2}|\nabla c|^2\right)+\chi\left(|D^2 c|^2+|\nabla c|^2\right)=-\chi n\Delta c+\lambda n-\mu n^2.
\end{align*}
Multiplying the above equality by $\phi^R_{x_0}$ and integrating in terms of spatial variable, it yields that
\begin{align*}
&\frac{\dd}{\dd t} \int_{\RR^d}\left(n+\frac{\chi\tau}{2}|\nabla c|^2\right)\phi^R_{x_0}\dd x-\int_{\RR^d}\Delta(n+\frac{\chi}{2}|\nabla c|^2)\phi^R_{x_0}\dd x+\chi\int_{\RR^d}(|D^2 c|^2+|\nabla c|^2)\phi^R_{x_0} \dd x\\
=&-\chi\int_{\RR^d} n\Delta c \phi^R_{x_0}\dd x+\lambda\int_{\RR^d} n\phi^R_{x_0}\dd x-\mu \int_{\RR^d}n^2\phi^R_{x_0}\dd x.
\end{align*}
Performing integration by parts, noting the definition of $\phi^R_{x_0}$, we obtain that
\begin{align*}
\int_{\RR^d}\Delta\left(n+\frac{\chi}{2}|\nabla c|^2\right)\phi^R_{x_0}\dd x&=\int_{B_{2R}(x_0)}\left(n+\frac{\chi}{2}|\nabla c|^2\right){\Delta\phi^R_{x_0}}\dd x.
\end{align*}
Because there exists a family of balls $\displaystyle\{B_{R}(y_i)\}_{i=1}^{3^d }$ such that
\[B_{2R}(x_0)\subseteq \bigcup_{i=1}^{3^d}B_{R}(y_i),\]
we deduce that
\begin{align*}
\int_{B_{2R}(x_0)}\left(n+\frac{\chi}{2}|\nabla c|^2\right){\Delta \phi^R_{x_0}}\dd x&\le\frac{C}{R^2}\int_{B_{2R}(x_0)}\left(n+\frac{\chi}{2}|\nabla c|^2\right)\dd x\\
&\le \frac{C}{R^2}\sum_{i=1}^{3^d}\int_{B_{R}(y_i)}\left(n+\frac{\chi}{2}|\nabla c|^2\right)\dd x\\
&\le \frac{C3^d}{R^2}\left(\|n\|_{L^1_{\uloc}(\RR^d)}+\frac{\chi}{2}\|\nabla c\|^2_{L^2_{\uloc}(\RR^d)}\right).
\end{align*}
  Owning to $|\Delta c|^2\le d|D^2 c|^2$, we have by  H\"{o}lder's and Young's inequalities that
\begin{align*}
-\chi\int_{\RR^d} n\Delta c \phi^R_{x_0}\dd x\le& \frac{\chi}{d}\int_{\RR^d}|\Delta c|^2\phi^R_{x_0} \dd x+\frac{d\chi}{4}\int_{\RR^d}n^2\phi^R_{x_0}\dd x\\
\le &{\chi}\int_{\RR^d}|D^2 c|^2\phi^R_{x_0} \dd x+\frac{d\chi}{4}\int_{\RR^d}n^2\phi^R_{x_0}\dd x.
\end{align*}
Therefore, the above estimates imply that
\begin{align*}
&\frac{\dd}{\dd t} \int_{\RR^d}\left(n+\frac{\chi\tau}{2}|\nabla c|^2\right)\phi^R_{x_0} \dd x+\frac{2}{\tau}\int_{\RR^d}\left(n+\frac{\chi\tau}{2}|\nabla c|^2\right)\phi^R_{x_0} \dd x\\
\le& \frac{C3^d}{R^2}\left(\|n\|_{L^1_{\uloc}(\RR^d)}+\frac{\chi}{2}\|\nabla c\|^2_{L^2_{\uloc}(\RR^d)}\right)+\left(\lambda+\frac{2}{\tau}\right)
\int_{\RR^d}n\phi^R_{x_0}+\left(\frac{d\chi}{4}-\mu\right)\int_{\RR^d}n^2\phi^R_{x_0}\dd x.
\end{align*}
Since $n^2\ge \e n-\frac{\e^2}{4}$ for any $\e\ge 0$ and for $\mu>\frac{d\chi}{4}$, we have that
\begin{align*}
&\frac{\dd}{\dd t} \int_{\RR^d}\left(n+\frac{\chi\tau}{2}|\nabla c|^2\right)\phi^R_{x_0} \dd x+\frac{2}{\tau}\int_{\RR^d}\left(n+\frac{\chi\tau}{2}|\nabla c|^2\right)\phi^R_{x_0} \dd x\\
\le& \frac{C3^d}{R^2}\left(\|n\|_{L^1_{\uloc}(\RR^d)}+\frac{\chi}{2}\|\nabla c\|^2_{L^2_{\uloc}(\RR^d)}\right)+\left(\lambda+\frac{2}{\tau}-\e\left(\mu-\frac{d\chi}{4}\right)\right)\int_{\RR^d}n\phi^R_{x_0}\dd x\\&+\left(\mu-\frac{d\chi}{4}\right)\frac{\e^2}{4}\int_{\RR^d}\phi^R_{x_0}\dd x\\
\le& \frac{C3^d}{R^2}(\|n\|_{L^1_{\uloc}(\RR^d)}+\frac{\chi}{2}\|\nabla c\|^2_{L^2_{\uloc}(\RR^d)})+\left(\lambda+\frac{2}{\tau}-
\e\left(\mu-\frac{d\chi}{4}\right)\right)\|n\|_{L^1_{\uloc}(\RR^d)}\\&+C\left(\mu-\frac{d\chi}{4}\right)\frac{\e^2}{4}R^d.
\end{align*}
For any $t\in [0, T]$, the above estimate shows that
\begin{align*}
&\frac{\dd}{\dd t} \int_{\RR^d}\left(n+\frac{\chi\tau}{2}|\nabla c|^2\right)\phi^R_{x_0} \dd x+\frac{2}{\tau}\int_{\RR^d}\left(n+\frac{\chi\tau}{2}|\nabla c|^2\right)\phi^R_{x_0} \dd x\\
\le& \frac{C3^d}{R^2}\left(\|n\|_{L^{\infty}_TL^1_{\uloc}(\RR^d)}+\frac{\chi}{2}\|\nabla c\|^2_{L^{\infty}_TL^2_{\uloc}(\RR^d)}\right)+\left(\lambda+\frac{2}{\tau}
-\e\left(\mu-\frac{d\chi}{4}\right)\right)\|n\|_{L^\infty_TL^1_{\uloc}(\RR^d)}\\&
+C\left(\mu-\frac{d\chi}{4}\right)\frac{\e^2}{4}R^d.
\end{align*}
Using the fact that $y'(t)+cy(t)\le C$ in $(0, T]$ implies
$$y\le \max\left\{y(0),\, \frac{C}{c}\right\}\,\, \text{ for all }t\in (0, T],$$ from the above inequality we obtain that, for any $t\in (0, T]$,
\begin{align*}
&\int_{\RR^d} n(t)\phi^R_{x_0} \dd x +\frac{\chi\tau}{2}\int_{\RR^d}|\nabla c(t)|^2\phi^R_{x_0} \dd x\\
\le &\max\Bigg\{\int_{\RR^d}n_0\phi^R_{x_0} \dd x +\frac{\chi\tau}{2}\int_{\RR^d}|\nabla c_0|^2\phi^R_{x_0} \dd x,\,\,\frac{C3^d\tau}{2R^2}\left(\|n\|_{L^{\infty}_TL^1_{\uloc}(\RR^d)}+\frac{\chi}{2}\|\nabla c\|^2_{L^{\infty}_TL^2_{\uloc}(\RR^d)}\right)\\
&\qquad\quad+\frac{\tau}{2}\left(\lambda+\frac{2}{\tau}-\e\left(\mu-\frac{d\chi}{4}\right)\right)
\|n\|_{L^\infty_TL^1_{\uloc}(\RR^d)}+\frac{C\tau}{2}\left(\mu-\frac{d\chi}{4}\right)\frac{\e^2}{4}R^d \Bigg\}.
\end{align*}
Noting the definition of $L^p_{\uloc}(\RR^d)$ and the support property of $\phi^R_{x_0}$, we infer from the above inequality that for any $0<t\le T$,
\begin{align*}
&\|n(t)\|_{L^1_{\uloc}(\RR^d)}+\frac{\chi\tau}{2}\|\nabla c(t)\|^2_{L^2_{\uloc}(\RR^d)}\\
\le& \max\Bigg\{2\|n_0\|_{L^1_{\uloc}(\RR^d)}+{\chi\tau}\|\nabla c_0\|^2_{L^2_{\uloc}(\RR^d)},\,\,
 \frac{C3^d\tau}{R^2}\left(\|n\|_{L^{\infty}_TL^1_{\uloc}(\RR^d)}+\frac{\chi}{2}\|\nabla c\|^2_{L^{\infty}_TL^2_{\uloc}(\RR^d)}\right)\\&\qquad\quad+{\tau}\left(\lambda+\frac{2}{\tau}
 -\e\left(\mu-\frac{d\chi}{4}\right)\right)\|n\|_{L^\infty_TL^1_{\uloc}(\RR^d)}+{C\tau}
 \left(\mu-\frac{d\chi}{4}\right)\frac{\e^2}{4}R^d \Bigg\}.
\end{align*}
Taking sup with respect to time variable $t$,  we get that
\begin{align*}
&\|n\|_{L^\infty_TL^1_{\uloc}(\RR^d)}+\frac{\chi\tau}{2}\|\nabla c\|^2_{L^\infty_TL^2_{\uloc}(\RR^d)}\\
\le& 2\max\Bigg\{2\|n_0\|_{L^1_{\uloc}(\RR^d)}+{\chi\tau}\|\nabla c_0\|^2_{L^2_{\uloc}(\RR^d)},\,\,
 \frac{C3^d\tau}{R^2}\left(\|n\|_{L^{\infty}_TL^1_{\uloc}(\RR^d)}+\frac{\chi}{2}\|\nabla c\|^2_{L^{\infty}_TL^2_{\uloc}(\RR^d)}\right)\\&\qquad\quad\,\,+{\tau}\left(\lambda+
 \frac{2}{\tau}-\e\left(\mu-\frac{d\chi}{4}\right)\right)\|n\|_{L^\infty_TL^1_{\uloc}(\RR^d)}+{C\tau}\left(\mu-\frac{d\chi}{4}\right)\frac{\e^2}{4}R^d \Bigg\}\\
=&\max\Bigg\{4\|n_0\|_{L^1_{\uloc}(\RR^d)}+2{\chi\tau}\|\nabla c_0\|^2_{L^2_{\uloc}(\RR^d)},\,\,
 \frac{C3^d\chi\tau}{R^2}\|\nabla c\|^2_{L^{\infty}_TL^2_{\uloc}(\RR^d)}+{C\tau}\left(\mu-\frac{d\chi}{4}\right)\frac{\e^2}{2}R^d \\&\qquad\qquad\qquad\qquad\qquad\qquad\,\,\,\,+2{\tau}\left(\frac{C3^d}{R^2}+\lambda+\frac{2}{\tau}-
 \e\left(\mu-\frac{d\chi}{4}\right)\right)\|n\|_{L^\infty_TL^1_{\uloc}(\RR^d)}\Bigg\}.
\end{align*}
We choosing some $R_0>1$, such that for all $R\ge R_0$, $$\frac{C3^d}{R^2}\le\frac{1}{4}.$$ As long as $\mu>\frac{d\chi}{4}$, we can choose some $\e$ such that
\[\frac{C3^d}{R^2}+\lambda+\frac{2}{\tau}-\e\left(\mu-\frac{d\chi}{4}\right)=0.\]
Then we infer that
\begin{align*}
&\|n\|_{L^\infty_TL^1_{\uloc}(\RR^d)}+\frac{\chi\tau}{2}\|\nabla c\|^2_{L^\infty_TL^2_{\uloc}(\RR^d)}\\
\le& \max\left\{4\|n_0\|_{L^1_{\uloc}(\RR^d)}+2{\chi\tau}\|\nabla c_0\|^2_{L^2_{\uloc}(\RR^d)},\,\frac{\chi\tau}{4}\|\nabla c\|^2_{L^{\infty}_TL^2_{\uloc}(\RR^d)}+C(\lambda, \tau, \mu, \chi, d, R) \right\}.
\end{align*}
Hence we can conclude that
\begin{align*}
&\|n\|_{L^\infty_TL^1_{\uloc}(\RR^d)}+\frac{\chi\tau}{4}\|\nabla c\|^2_{L^\infty_TL^2_{\uloc}(\RR^d)}\\
\le &4\|n_0\|_{L^1_{\uloc}(\RR^d)}+2{\chi\tau}\|\nabla c_0\|^2_{L^2_{\uloc}(\RR^d)}+C(\lambda, \tau, \mu, \chi, d, R).
\end{align*}
We complete the proof.
\end{proof}
\begin{proposition}\label{lemmank}Let $\lambda, \,\chi>0$, $R\ge 1$ and $\mu\ge0$. Then for any $k\in\NN$ with $k\ge 2$, one can find some absolute constant $C$ and $C_k$ depending on $k,\lambda,\chi$ such that the solution $(n,c)$ of equations \eqref{che-M}-\eqref{che-I-M} satisfies
\begin{equation}\label{nlock}
\begin{aligned}
&\frac{\dd}{\dd t}\int_{\RR^d}n^k\phi^R_{x_0}\dd x+\frac{k(k-1)}{4}\int_{\RR^d}|\nabla n|^2 n^{k-2}\phi^R_{x_0}\dd x\\
\le&\frac{C3^dk}{2(k-1)R^2}\|n\|^k_{L^k_{\textnormal{uloc}}(\RR^d)}+\frac{C3^dk}{R^{2k}}\|\nabla c\|^{2k}_{L^{2k}_{\textnormal{uloc}}(\RR^d)}+k\int_{\RR^d}n^{2}|\nabla c|^{2k-2}\phi^R_{x_0}\dd x\\
&+(C_k-\mu k)\int_{\RR^d} n^{k+1}\phi^R_{x_0}\dd x+C(\lambda+1) R^dk.
\end{aligned}
\end{equation}
\end{proposition}
\begin{proof}Multiplying E.q.$\eqref{che}_1$ by $n^{k-1}\phi^R_{x_0}$ and then integrating with respect to spatial variable, we obtain that
\begin{align*}
&\frac{1}{k}\frac{\dd}{\dd t}\int_{\RR^d}n^k\phi^R_{x_0}\dd x-\int_{\RR^d}\Delta n n^{k-1}\phi^R_{x_0}\dd x\\
=&-\chi\int_{\RR^d}\nabla \cdot(n\nabla c)n^{k-1}\phi^R_{x_0}\dd x+\lambda\int_{\RR^d} n^{k}\phi^R_{x_0}\dd x-\mu\int_{\RR^d} n^{k+1}\phi^R_{x_0}\dd x.
\end{align*}
Integration by parts  yields
\begin{align*}
-\int_{\RR^d}\Delta n n^{k-1}\phi^R_{x_0}\dd x=(k-1)\int_{\RR^d}|\nabla n|^2 n^{k-2}\phi^R_{x_0}\dd x+\int_{\RR^d}n^{k-1}\nabla n\cdot\nabla\phi^R_{x_0}\dd x,
\end{align*}
and
\begin{align*}
-\chi\int_{\RR^d}\nabla \cdot(n\nabla c)n^{k-1}\phi^R_{x_0}\dd x=\chi\int_{\RR^d}n^k\nabla c\cdot\nabla\phi^R_{x_0}\dd x+\chi(k-1)\int_{\RR^d}n^{k-1}\nabla c\cdot\nabla n\,\phi^R_{x_0}\dd x.
\end{align*}
Therefore, from the above two equalities, we obtain that
\begin{equation}\label{nk}
\begin{aligned}
&\frac{1}{k}\frac{\dd}{\dd t}\int_{\RR^d}n^k\phi^R_{x_0}\dd x+(k-1)\int_{\RR^d}|\nabla n|^2 n^{k-2}\phi^R_{x_0}\dd x\\
=&-\int_{\RR^d}n^{k-1}\nabla n\cdot\nabla\phi^R_{x_0}\dd x+\chi\int_{\RR^d}n^k\nabla c\cdot\nabla\phi^R_{x_0}\dd x+\lambda\int_{\RR^d} n^{k}\phi^R_{x_0}\dd x\\
&-\mu\int_{\RR^d} n^{k+1}\phi^R_{x_0}\dd x+\chi(k-1)\int_{\RR^d}n^{k-1}\nabla c\cdot\nabla n\,\phi^R_{x_0}\dd x.
\end{aligned}
\end{equation}
By virtue of H\"{o}lder's inequality and Young's inequality, we achieve that
\begin{align*}
-\int_{\RR^d}n^{k-1}\nabla n\cdot\nabla\phi^R_{x_0}\dd x\le \frac{k-1}{2}\int_{\RR^d}|\nabla n|^2 n^{k-2}\phi^R_{x_0}\dd x+\frac{1}{2(k-1)}\int_{\RR^d}n^k\frac{|\nabla \phi^R_{x_0}|^2}{\phi^R_{x_0}}\dd x.
\end{align*}
According to the expression of $\phi^R_{x_0}$, an easy computation yields that for $x\in B_{2R}(x_0)$,
\begin{align*}
&\phi^{-1}_{x_0, R}|\nabla \phi_{x_0, R}|^2\\=&\left(\frac{|x-x_0|}{4R^2}\right)^2\left( \frac{4R^2}{|x-x_0|^2-4R^2}\right)^4\exp\left(\frac{4}{3}+{\frac{4R^2}{|x-x_0|^2-4R^2}}\right)
\le \frac{C}{R^2}.
\end{align*}
Hence, it easy to check that
\begin{equation}\label{1}
\begin{aligned}
-\int_{\RR^d}n^{k-1}\nabla n\cdot\nabla\phi^R_{x_0}\dd x&\le \frac{k-1}{2}\int_{\RR^d}|\nabla n|^2 n^{k-2}\phi^R_{x_0}\dd x+\frac{C}{2(k-1)R^2}\int_{B_{2R}(x_0)}n^k\dd x\\
&\le\frac{k-1}{2}\int_{\RR^d}|\nabla n|^2 n^{k-2}\phi^R_{x_0}\dd x+\frac{C3^d}{2(k-1)R^2}\|n\|^k_{L^k_{\uloc}(\RR^d)}.
\end{aligned}
\end{equation}
Similarly, by H\"{o}lder's inequality and Young's inequality, using the definition of $\phi^R_{x_0}$, we infer that
\begin{equation}\label{2}
\begin{aligned}
&\chi\int_{\RR^d}n^k\nabla c\cdot\nabla\phi^R_{x_0}\dd x\\
\le& \chi^{\frac{k+1}{k}}\int_{\RR^d}n^{k+1}\phi^R_{x_0}\dd x+\int_{\RR^d}|\nabla c|^{2k}( \phi^R_{x_0})^{-\frac{k}{k+1}\cdot 2k}|\nabla \phi^R_{x_0}|^{2k}\dd x+\Big(\int_{B_{2R}(x_0)}1\dd x\Big)^{\frac{k-1}{2k(k+1)}}\\
\le& \chi^{\frac{k+1}{k}}\int_{\RR^d}n^{k+1}\phi^R_{x_0}\dd x+\frac{C}{R^{2k}}\int_{B_{2R}(x_0)}|\nabla c|^{2k}\dd x+CR^{\frac{k-1}{2k(k+1)}}\\
\le& \chi^{\frac{k+1}{k}}\int_{\RR^d}n^{k+1}\phi^R_{x_0}\dd x+\frac{C3^d}{R^{2k}}\|\nabla c\|^{2k}_{L^{2k}_{\uloc}(\RR^d)}+CR^{\frac{k-1}{2k(k+1)}}.
\end{aligned}
\end{equation}
Also,  the third term on the right-hand side of equality \eqref{nk} can be bounded  as follows
\begin{equation}\label{3}
\begin{aligned}
&\chi(k-1)\int_{\RR^d}n^{k-1}\nabla c\cdot\nabla n\,\phi^R_{x_0}\dd x\\
\le& \frac{k-1}{4}\int_{\RR^d}|\nabla n|^{2}n^{k-2}\phi^R_{x_0}\dd x+\chi^2(k-1)\int_{\RR^d}n^{k}|\nabla c|^2\phi^R_{x_0}\dd x\\
\le & \frac{k-1}{4}\int_{\RR^d}|\nabla n|^{2}n^{k-2}\phi^R_{x_0}\dd x+\int_{\RR^d}n^{2}|\nabla c|^{2k-2}\phi^R_{x_0}\dd x +\chi^{\frac{2(k-1)}{k-2}}(k-1)^{\frac{k-1}{k-2}}\int_{\RR^d}n^{k+1}\phi^R_{x_0}\dd x,
\end{aligned}
\end{equation}
in the third line, we have used the fact
\[\int_{\RR^d}n^{k}|\nabla c|^2\phi^R_{x_0}\dd x=\int_{\RR^d}\Big(n^{2}|\nabla c|^{2k-2}\Big)^{\frac{1}{k-1}}n^{\frac{(k-2)(k+1)}{k-1}}\phi^R_{x_0}\dd x.\]
Substituting inequalities \eqref{1}--\eqref{3} into equality \eqref{nk}, we obtain that
\begin{align*}
&\frac{1}{k}\frac{\dd}{\dd t}\int_{\RR^d}n^k\phi^R_{x_0}\dd x+\frac{k-1}{4}\int_{\RR^d}|\nabla n|^2 n^{k-2}\phi^R_{x_0}\dd x\\
\le&\frac{C3^d}{2(k-1)R^2}\|n\|^k_{L^k_{\uloc}(\RR^d)}+\frac{C3^d}{R^{2k}}\|\nabla c\|^{2k}_{L^{2k}_{\uloc}(\RR^d)}+CR^{\frac{k-1}{2k(k+1)}}+\int_{\RR^d}n^{2}|\nabla c|^{2k-2}\phi^R_{x_0}\dd x\\
&+\lambda\int_{\RR^d} n^{k}\phi^R_{x_0}\dd x+\left(\chi^{\frac{k+1}{k}}+\chi^{\frac{2(k-1)}{k-2}}(k-1)^{\frac{k-1}{k-2}}-\mu\right)\int_{\RR^d} n^{k+1}\phi^R_{x_0}\dd x.
\end{align*}
By H\"{o}lder's inequality, it yields that
\begin{align*}
&\lambda\int_{\RR^d} n^{k}\phi_{x_0, R}\dd x+\left(\chi^{\frac{k+1}{k}}+\chi^{\frac{2(k-1)}{k-2}}(k-1)^{\frac{k-1}{k-2}}-\mu\right)\int_{\RR^d} n^{k+1}\phi^R_{x_0}\dd x\\
\le&\left(\lambda+\chi^{\frac{k+1}{k}}+\chi^{\frac{2(k-1)}{k-2}}(k-1)^{\frac{k-1}{k-2}}-\mu\right)\int_{\RR^d} n^{k+1}\phi^R_{x_0}\dd x+C\lambda R^d.
\end{align*}
Hence, we conclude Proposition \ref{lemmank}.
\end{proof}
From Proposition \ref{lemmank}, we need to establish estimates on $\|\nabla c\|^{2k}_{L^{2k}_{\uloc}(\RR^d)}$ and $$\int_{\RR^d}n^{2}|\nabla c|^{2k-2}\phi^R_{x_0}\dd x.$$ These can be done by the following propositions.
\begin{proposition}Let $\tau>0$, $R\ge 1$ and $(n,c)$ be the solution of system \eqref{che-M}-\eqref{che-I-M}. For any $k\in\NN$ with $k\ge 2$, there exists an absolute constant $C$ such that
\begin{equation}\label{nablac2k}
\begin{aligned}
&\frac{\dd}{\dd t}\int_{\RR^d}|\nabla c|^{2k}\phi^R_{x_0}\dd x+\frac{k(k-1)}{4\tau}\int_{\RR^d}|\nabla |\nabla c|^2|^2|\nabla c|^{2k-4}\phi^R_{x_0}\dd x\\
&+\frac{k}{\tau}\int_{\RR^d}|D^2 c|^2|\nabla c|^{2k-2}\phi^R_{x_0}\dd x+\frac{2k}{\tau}\int_{\RR^d}|\nabla c|^{2k}\phi^R_{x_0}\dd x\\
\le&\frac{(d+1+2(k-1))k}{\tau}\int_{\RR^d} n^2|\nabla c|^{2k-2}\phi^R_{x_0}\dd x+\frac{C3^dk}{\tau R^2}\|\nabla c\|^{2k}_{L^{2k}_{\uloc}(\RR^d)}.
\end{aligned}
\end{equation}
\end{proposition}
\begin{proof}Taking $\nabla$ on E.q.$\eqref{che}_2$ and then taking inner product by $\nabla c|\nabla c|^{2k-2}\phi^R_{x_0}$, one has
\begin{align*}
&\frac{1}{2k}\frac{\dd}{\dd t}\int_{\RR^d}|\nabla c|^{2k}\phi^R_{x_0}\dd x-\frac{1}{\tau}\int_{\RR^d}\nabla\Delta c\cdot\nabla c|\nabla c|^{2k-2}\phi^R_{x_0}\dd x+\frac{1}{\tau}\int_{\RR^d}|\nabla c|^{2k}\phi^R_{x_0}\dd x\\
=&\frac{1}{\tau}\int_{\RR^d}\nabla n\cdot\nabla c|\nabla c|^{2k-2}\phi^R_{x_0}\dd x.
\end{align*}
By virtue of $$\nabla\Delta c\cdot\nabla c=-\frac{1}{2}\Delta |\nabla c|^2+|D^2 c|^2,$$ the above equality implies that
\begin{align*}
&\frac{1}{2k}\frac{\dd}{\dd t}\int_{\RR^d}|\nabla c|^{2k}\phi^R_{x_0}\dd x
 +\frac{1}{\tau}\int_{\RR^d}|D^2 c|^2|\nabla c|^{2k-2}\phi^R_{x_0}\dd x+\frac{1}{\tau}\int_{\RR^d}|\nabla c|^{2k}\phi^R_{x_0}\dd x
\\= &\frac{1}{\tau}\int_{\RR^d}\nabla n\cdot\nabla c|\nabla c|^{2k-2}\phi^R_{x_0}\dd x+\frac{1}{2\tau}\int_{\RR^d}\Delta |\nabla c|^2|\nabla c|^{2k-2}\phi^R_{x_0}\dd x.
\end{align*}
By integration by parts, it yields that
\begin{align*}
\begin{split}
&-\frac{1}{2\tau}\int_{\RR^d}\Delta |\nabla c|^2|\nabla c|^{2k-2}\phi^R_{x_0}\dd x\\
=&\frac{k-1}{2\tau}\int_{\RR^d}|\nabla |\nabla c|^2|^2|\nabla c|^{2k-4}\phi^R_{x_0}\dd x
+\frac{1}{2\tau}\int_{\RR^d}\nabla |\nabla c|^2\cdot\nabla\phi^R_{x_0}|\nabla c|^{2k-2}\dd x,
\end{split}
\end{align*}
where Young's inequality gives
\begin{align*}
\begin{split}
&\Big|\frac{1}{2\tau}\int_{\RR^d}\nabla |\nabla c|^2\cdot\nabla\phi^R_{x_0}|\nabla c|^{2k-2}\dd x
\Big|\\
\le&\frac{k-1}{4\tau}\int_{\RR^d}|\nabla |\nabla c|^2|^2|\nabla c|^{2k-4}\phi^R_{x_0}\dd x+\frac{1}{4\tau(k-1)}\int_{\RR^d}|\nabla c|^{2k}(\phi^R_{x_0})^{-1}|\nabla\phi^R_{x_0}|^{2}\dd x\\
\le&\frac{k-1}{4\tau}\int_{\RR^d}|\nabla |\nabla c|^2|^2|\nabla c|^{2k-4}\phi^R_{x_0}\dd x+\frac{C3^d}{4\tau(k-1)R^2}\|\nabla c\|^{2k}_{L^{2k}_{\uloc}(\RR^d)}.
\end{split}
\end{align*}
By integration by parts and Young's inequality, one gets that
\begin{align*}
&\frac{1}{\tau}\int_{\RR^d}\nabla n\cdot\nabla c|\nabla c|^{2k-2}\phi^R_{x_0}\dd x\\
=&-\frac{1}{\tau}\int_{\RR^d} n\Delta c|\nabla c|^{2k-2}\phi^R_{x_0}\dd x-\frac{k-1}{\tau}\int_{\RR^d} n\nabla c\cdot\nabla|\nabla c|^2|\nabla c|^{2k-4}\phi^R_{x_0}\dd x\\
&-\frac{1}{\tau}\int_{\RR^d} n\nabla c\cdot\nabla \phi^R_{x_0}|\nabla c|^{2k-2}\dd x\\
\le &\frac{1}{2d\tau}\int_{\RR^d} |\Delta c|^2|\nabla c|^{2k-2}\phi^R_{x_0}\dd x+\frac{d}{2\tau}\int_{\RR^d} n^2|\nabla c|^{2k-2}\phi^R_{x_0}\dd x\\&+\frac{k-1}{8\tau}\int_{\RR^d}|\nabla |\nabla c|^2|^2|\nabla c|^{2k-4}\phi^R_{x_0}\dd x+\frac{2(k-1)}{\tau}\int_{\RR^d}n^2|\nabla c|^{2k-2}\phi^R_{x_0}\dd x
\\&+\frac{1}{2\tau}\int_{\RR^d} n^2|\nabla c|^{2k-2}\phi^R_{x_0}\dd x+\frac{1}{2\tau}\int_{\RR^d}|\nabla c|^{2k}(\phi^R_{x_0})^{-1}|\nabla\phi^R_{x_0}|^2\dd x\\
\le &\frac{1}{2\tau}\int_{\RR^d} |D^2 c|^2|\nabla c|^{2k-2}\phi^R_{x_0}\dd x+\frac{d+1+4(k-1)}{2\tau}\int_{\RR^d} n^2|\nabla c|^{2k-2}\phi^R_{x_0}\dd x\\&+\frac{C3^d}{\tau R^2}\|\nabla c\|^{2k}_{L^{2k}_{\uloc}(\RR^d)}.
\end{align*}
To sum up, it is easy to check that
\begin{align*}
&\frac{1}{2k}\frac{\dd}{\dd t}\int_{\RR^d}|\nabla c|^{2k}\phi^R_{x_0}\dd x+\frac{k-1}{8\tau}\int_{\RR^d}|\nabla |\nabla c|^2|^2|\nabla c|^{2k-4}\phi^R_{x_0}\dd x\\
&+\frac{1}{2\tau}\int_{\RR^d}|D^2 c|^2|\nabla c|^{2k-2}\phi^R_{x_0}\dd x+\frac{1}{\tau}\int_{\RR^d}|\nabla c|^{2k}\phi^R_{x_0}\dd x\\
\le&\frac{d+1+4(k-1)}{2\tau}\int_{\RR^d} n^2|\nabla c|^{2k-2}\phi^R_{x_0}\dd x+\frac{C3^d}{\tau R^2}\|\nabla c\|^{2k}_{L^{2k}_{\uloc}(\RR^d)}.
\end{align*}
We end the proof of the proposition.
\end{proof}
In order to absorb the term on the right-hand side of inequalities \eqref{nlock} and \eqref{nablac2k}, making use of logistic term $\lambda n-\mu n^2$, this may be tackled by the following inequality.
\begin{proposition}\label{n1c2k-2}
Assume $\tau, \lambda, \chi>0, \mu\ge0$ and $R\ge 1$. Fixed $k\in\NN$ with $k\ge 2$, there exist a constant $C_1(\chi,\tau,\lambda, k)$ such that the solution $(n,c)$ of equations \eqref{che-M}-\eqref{che-I-M} fulfills the following estimate
\begin{align*}
&\frac{\dd}{\dd t}\int_{\RR^d}n|\nabla c|^{2k-2}\phi^R_{x_0}\dd x+\frac{(k-1)(k-2)}{2\tau}\int_{\RR^d}|\nabla|\nabla c|^2|^2|\nabla c|^{2k-6}n\phi^R_{x_0}\dd x\\
&+\frac{2k-2}{\tau}\int_{\RR^d}|D^2 c|^2|\nabla c|^{2k-4}n\phi^R_{x_0}\dd x\\
\le&C_1\int_{\RR^d}|\nabla|\nabla c|^2|^2|\nabla c|^{2k-4}\phi^R_{x_0}\dd x
+\frac{\lambda}{2}\int_{\RR^d}|\nabla c|^{2k-2}\phi^R_{x_0}\dd x+\frac{C3^d(1+\frac{1}{\tau})}{ R^{2}}\|\nabla c\|^{2k}_{L^{2k}_{\uloc}(\RR^d)}\\
&+\frac{C3^d}{\tau R^{2}}\|n\|^{k}_{L^{k}_{\uloc}(\RR^d)}+(C_1-\mu)
\int_{\RR^d}n^2|\nabla c|^{2k-2}\phi^R_{x_0}\dd x+\int_{\RR^d}|\nabla n|^2|\nabla c|^{2k-4}\phi^R_{x_0}\dd x.
\end{align*}
\end{proposition}
\begin{proof}By equations \eqref{che}, it yields that
\begin{equation}\label{nc2k}
\begin{aligned}
&\frac{\dd}{\dd t}\int_{\RR^d}n|\nabla c|^{2k-2}\phi^R_{x_0}\dd x+\frac{2k-2}{\tau}\int_{\RR^d}n|\nabla c|^{2k-2}\phi^R_{x_0}\dd x\\
=&\int_{\RR^d}\Delta n|\nabla c|^{2k-2}\phi^R_{x_0}\dd x-\chi\int_{\RR^d}\nabla\cdot(n\nabla c)|\nabla c|^{2k-2}\phi^R_{x_0}\dd x+\lambda \int_{\RR^d}n|\nabla c|^{2k-2}\phi^R_{x_0}\dd x\\
&-\mu \int_{\RR^d}n^2|\nabla c|^{2k-2}\phi^R_{x_0}\dd x+\frac{2k-2}{\tau}\int_{\RR^d}\nabla\Delta c\cdot\nabla c|\nabla c|^{2k-4}n\phi^R_{x_0}\dd x
\\
&+\frac{2k-2}{\tau}\int_{\RR^d}\nabla n\cdot\nabla c|\nabla c|^{2k-4}n\phi^R_{x_0}\dd x.
\end{aligned}
\end{equation}
By virtue of integration by parts, we obtain that
\begin{equation}\label{nc2k-21}
\begin{aligned}
&-\chi\int_{\RR^d}\nabla\cdot(n\nabla c)|\nabla c|^{2k-2}\phi^R_{x_0}\dd x\\
=&\chi(k-1)\int_{\RR^d}n\nabla c\cdot\nabla|\nabla c|^2|\nabla c|^{2k-4}\phi^R_{x_0}\dd x
+\chi\int_{\RR^d}n\nabla c\cdot\nabla\phi^R_{x_0}|\nabla c|^{2k-2}\dd x\\
\le&\frac{(k-1)^2}{2\tau^2}\int_{\RR^d}|\nabla|\nabla c|^2|^2|\nabla c|^{2k-4}\phi^R_{x_0}\dd x+\frac{\chi^2\tau^2}{2}\int_{\RR^d}n^2|\nabla c|^{2k-2}\phi^R_{x_0}\dd x\\
&+\frac{\chi^2}{2}\int_{\RR^d}n^2|\nabla c|^{2k-2}\phi^R_{x_0}\dd x
+\frac{1}{2}\int_{\RR^d}|\nabla c|^{2k}(\phi^R_{x_0})^{-1}|\nabla \phi^R_{x_0}|^2\dd x\\
\le&\frac{(k-1)^2}{2\tau^2}\int_{\RR^d}|\nabla|\nabla c|^2|^2|\nabla c|^{2k-4}\phi^R_{x_0}\dd x+\frac{\chi^2(1+\tau^2)}{2}\int_{\RR^d}n^2|\nabla c|^{2k-2}\phi^R_{x_0}\dd x
\\&+\frac{C3^d}{R^2}\|\nabla c\|^{2k}_{L^{2k}_{\uloc}(\RR^d)}.
\end{aligned}
\end{equation}
Following the above methods, we can infer that
\begin{align*}
&\frac{2k-2}{\tau}\int_{\RR^d}\nabla\Delta c\cdot\nabla c|\nabla c|^{2k-4}n\phi^R_{x_0}\dd x\\
=&\frac{k-1}{\tau}\int_{\RR^d}\Delta|\nabla c|^2|\nabla c|^{2k-4}n\phi^R_{x_0}\dd x-\frac{2k-2}{\tau}\int_{\RR^d}|D^2 c|^2|\nabla c|^{2k-4}n\phi^R_{x_0}\dd x\\
=&-\frac{(k-1)(k-2)}{\tau}\int_{\RR^d}|\nabla|\nabla c|^2|^2|\nabla c|^{2k-6}n\phi^R_{x_0}\dd x-\frac{k-1}{\tau}\int_{\RR^d}\nabla|\nabla c|^2\cdot\nabla n|\nabla c|^{2k-4}\phi^R_{x_0}\dd x\\
&-\frac{k-1}{\tau}\int_{\RR^d}\nabla|\nabla c|^2\cdot\nabla \phi^R_{x_0}|\nabla c|^{2k-4}n\dd x
-\frac{2k-2}{\tau}\int_{\RR^d}|D^2 c|^2|\nabla c|^{2k-4}n\phi^R_{x_0}\dd x,
\end{align*}
where H\"{o}lder's and Young's inequalities yields that
\begin{align*}
&-\frac{k-1}{\tau}\int_{\RR^d}\nabla|\nabla c|^2\cdot\nabla n|\nabla c|^{2k-4}\phi^R_{x_0}\dd x\\
\le&\frac{1}{2}\int_{\RR^d}|\nabla n|^2|\nabla c|^{2k-4}\phi^R_{x_0}\dd x+\frac{(k-1)^2}{2\tau^2}\int_{\RR^d}|\nabla|\nabla c|^2|^2|\nabla c|^{2k-4}\phi^R_{x_0}\dd x,
\end{align*}
and
\begin{align*}
&-\frac{k-1}{\tau}\int_{\RR^d}\nabla|\nabla c|^2\cdot\nabla \phi^R_{x_0}|\nabla c|^{2k-4}n\dd x\\
\le&\frac{(k-1)(k-2)}{2\tau}\int_{\RR^d}|\nabla|\nabla c|^2|^2|\nabla c|^{2k-6}n\phi^R_{x_0}\dd x
+\frac{(k-1)}{2\tau(k-2)}\int_{\RR^d}|\nabla c|^{2k-2}n(\phi^R_{x_0})^{-1}|\nabla \phi^R_{x_0}|^2\dd x\\
\le&\frac{(k-1)(k-2)}{2\tau}\int_{\RR^d}|\nabla|\nabla c|^2|^2|\nabla c|^{2k-6}n\phi^R_{x_0}\dd x
+\frac{C3^d}{\tau R^{2}}\|\nabla c\|^{2k}_{L^{2k}_{\uloc}(\RR^d)}+\frac{C3^d}{\tau R^{2}}\|n\|^{k}_{L^{k}_{\uloc}(\RR^d)}.
\end{align*}
Hence, we conclude that
\begin{equation}\label{nc2k-22}
\begin{aligned}
&\frac{2k-2}{\tau}\int_{\RR^d}\nabla\Delta c\cdot\nabla c|\nabla c|^{2k-4}n\phi^R_{x_0}\dd x\\
\le&-\frac{(k-1)(k-2)}{2\tau}\int_{\RR^d}|\nabla|\nabla c|^2|^2|\nabla c|^{2k-6}n\phi^R_{x_0}\dd x+\frac{1}{2}\int_{\RR^d}|\nabla n|^2|\nabla c|^{2k-4}\phi^R_{x_0}\dd x\\
&+\frac{(k-1)^2}{2\tau^2}\int_{\RR^d}|\nabla|\nabla c|^2|^2|\nabla c|^{2k-4}\phi^R_{x_0}\dd x-\frac{2k-2}{\tau}\int_{\RR^d}|D^2 c|^2|\nabla c|^{2k-4}n\phi^R_{x_0}\dd x\\
&+\frac{C3^d}{\tau R^{2}}\left(\|\nabla c\|^{2k}_{L^{2k}_{\uloc}(\RR^d)}+\|n\|^{k}_{L^{k}_{\uloc}(\RR^d)}\right).
\end{aligned}
\end{equation}
Similarly, it is easy to check that
\begin{align*}
&\int_{\RR^d}\Delta n|\nabla c|^{2k-2}\phi^R_{x_0}\dd x+\frac{2k-2}{\tau}\int_{\RR^d}\nabla n\cdot\nabla c|\nabla c|^{2k-4}n\phi^R_{x_0}\dd x\\
\le&-\int_{\RR^d}\nabla n \cdot\nabla \phi^R_{x_0}|\nabla c|^{2k-2}\dd x-(k-1)\int_{\RR^d}\nabla n \cdot\nabla |\nabla c|^2 |\nabla c|^{2k-4}\phi^R_{x_0}\dd x\\
&+\frac{1}{4}\int_{\RR^d}|\nabla n|^2|\nabla c|^{2k-4}\phi^R_{x_0}\dd x+\frac{(2k-2)^2}{\tau^2}\int_{\RR^d}n^2|\nabla c|^{2k-2}\phi^R_{x_0}\dd x,
\end{align*}
moreover, we get by the H\"older inequality that
\begin{equation}\label{nc2k-23}
\begin{aligned}
&\int_{\RR^d}\Delta n|\nabla c|^{2k-2}\phi^R_{x_0}\dd x+\frac{2k-2}{\tau}\int_{\RR^d}\nabla n\cdot\nabla c|\nabla c|^{2k-4}n\phi^R_{x_0}\dd x\\
\le& \frac{1}{2}\int_{\RR^d}|\nabla n|^2|\nabla c|^{2k-4}\phi^R_{x_0}\dd x+2(k-1)^2\int_{\RR^d}|\nabla |\nabla c|^2|^2 |\nabla c|^{2k-4}\phi^R_{x_0}\dd x\\&+\frac{(2k-2)^2}{\tau^2}\int_{\RR^d}n^2|\nabla c|^{2k-2}\phi^R_{x_0}\dd x+\frac{C3^d}{R^2}\|\nabla c\|^{2k}_{L^{2k}_{\uloc}(\RR^d)}.
\end{aligned}
\end{equation}
Noting the fact that
\begin{align*}
\lambda\int_{\RR^d}n|\nabla c|^{2k-2}\phi^R_{x_0}\dd x\le \frac{\lambda}{2}\int_{\RR^d}n^2|\nabla c|^{2k-2}\phi^R_{x_0}\dd x
+\frac{\lambda}{2}\int_{\RR^d}|\nabla c|^{2k-2}\phi^R_{x_0}\dd x,
\end{align*}
combining with inequalities \eqref{nc2k-21}--\eqref{nc2k-23}, we deduce from \eqref{nc2k} that
\begin{align*}
&\frac{\dd}{\dd t}\int_{\RR^d}n|\nabla c|^{2k-2}\phi^R_{x_0}\dd x+\frac{(k-1)(k-2)}{2\tau}\int_{\RR^d}|\nabla|\nabla c|^2|^2|\nabla c|^{2k-6}n\phi^R_{x_0}\dd x\\
&+\frac{2k-2}{\tau}\int_{\RR^d}|D^2 c|^2|\nabla c|^{2k-4}n\phi^R_{x_0}\dd x\\
\le&\frac{(k-1)^2(1+2\tau^2)}{2\tau^2}\int_{\RR^d}|\nabla|\nabla c|^2|^2|\nabla c|^{2k-4}\phi^R_{x_0}\dd x
+\frac{\lambda}{2}\int_{\RR^d}|\nabla c|^{2k-2}\phi^R_{x_0}\dd x\\
&+\frac{C3^d}{\tau R^{2}}\|n\|^{k}_{L^{k}_{\uloc}(\RR^d)}+\left(\frac{\lambda+\chi^2(1+\tau^2)}{2}+\frac{(2k-2)^2}{\tau^2}-\mu\right)\int_{\RR^d}n^2|\nabla c|^{2k-2}\phi^R_{x_0}\dd x\\
&+\int_{\RR^d}|\nabla n|^2|\nabla c|^{2k-4}\phi^R_{x_0}\dd x+\frac{C3^d(1+\frac{1}{\tau})}{ R^{2}}\|\nabla c\|^{2k}_{L^{2k}_{\uloc}(\RR^d)}.
\end{align*}
We finish the proof of proposition.
\end{proof}
To digest the last term on the right-hand side of the inequality in Lemma \ref{n1c2k-2}, using the coupling structure of equations \eqref{che-M}-\eqref{che-I-M}, we shall to bound the following integral term
$$\int_{\RR^d} n^j|\nabla c|^{2k-2j}\dd x\,\,(j=2,3,\cdots,k-1)$$
as shown in the following proposition.
\begin{proposition}\label{njc2k-2j}Assume $\tau, \chi,\mu>0,\lambda\ge 0$ and $R\ge 1$. Let $k\in\NN$ with $k\ge 3$ and $j\in\NN$ with $2\le j\le k-1.$ Then there exist constant $C_j$ depending on $j,\lambda, \mu, \chi$ and an absolute constant $C$ such that the solution $(n, c)$ of system \eqref{che-M}-\eqref{che-I-M} satisfies
\begin{align*}
&\frac{\dd}{\dd t}\int_{\RR^d}n^j|\nabla c|^{2k-2j}\phi^R_{x_0}\dd x+\frac{j(j-1)}{4}\int_{\RR^d}n^{j-2}|\nabla n|^2|\nabla c|^{2k-2j}\phi^R_{x_0}\dd x\\
\le&\int_{\RR^d}n^{j-1}|\nabla n|^2|\nabla c|^{2k-2j-2}\phi^R_{x_0}\dd x+C_j\int_{\RR^d}|\nabla c|^{2k-4}|\nabla |\nabla c|^2|^2\phi^R_{x_0}\dd x\\
&+(C_j-\mu j)\int_{\RR^d}n^{j+1}|\nabla c|^{2k-2j}\phi^R_{x_0}\dd x+\lambda j\int_{\RR^d}|\nabla c|^{2k-2}\phi^R_{x_0}\dd x+C\lambda jR^{d}\\
&+C_j\int_{\RR^d}n^2|\nabla c|^{2k-2}\phi^R_{x_0}\dd x+\frac{C_j}{R^2}\|n\|^k_{L^k_{\uloc}(\RR^d)}
+\frac{C_j}{R^2}\|\nabla c\|^{2k}_{L^{2k}_{\uloc}(\RR^d)}.
\end{align*}
\end{proposition}
\begin{proof}
From system \eqref{che-M}, we can deduce that
\begin{align*}
&\frac{\dd}{\dd t}\int_{\RR^d}n^j|\nabla c|^{2k-2j}\phi^R_{x_0}\dd x+\frac{2(k-j)}{\tau}\int_{\RR^d}n^j|\nabla c|^{2k-2j}\phi^R_{x_0}\dd x\\
=&j\int_{\RR^d}\Delta n|\nabla c|^{2k-2j}n^{j-1}\phi^R_{x_0}\dd x-\chi j\int_{\RR^d}\nabla\cdot(n\nabla c)|\nabla c|^{2k-2j}n^{j-1}\phi^R_{x_0}\dd x\\
&+\lambda j\int_{\RR^d}n^j|\nabla c|^{2k-2j}\phi^R_{x_0}\dd x-\mu j\int_{\RR^d}n^{j+1}|\nabla c|^{2k-2j}\phi^R_{x_0}\dd x\\
&+\frac{2(k-j)}{\tau}\int_{\RR^d}\nabla\Delta c\cdot\nabla c n^j|\nabla c|^{2k-2j-2}\phi^R_{x_0}\dd x+\frac{2(k-j)}{\tau}\int_{\RR^d}n^j\nabla n\cdot\nabla c |\nabla c|^{2k-2j-2}\phi^R_{x_0}\dd x\\
:=&I_1+I_2+I_3+I_4+I_5+I_6.
\end{align*}
Because $$\nabla \Delta c\cdot\nabla c=\frac{1}{2}\Delta|\nabla c|^2 -|D^2 c|^2,$$ the term $I_5$ can be rewritten as follows.
\begin{align*}
I_5=&
\frac{(k-j)}{\tau}\int_{\RR^d}\Delta|\nabla c|^2n^j|\nabla c|^{2k-2j-2}\phi^R_{x_0}\dd x
-\frac{2(k-j)}{\tau}\int_{\RR^d}|D^2 c|^2 n^j|\nabla c|^{2k-2j-2}\phi^R_{x_0}\dd x\\
=
&-\frac{(k-j)j}{\tau}\int_{\RR^d}n^{j-1}\nabla n\cdot\nabla|\nabla c|^2|\nabla c|^{2k-2j-2}\phi^R_{x_0}\dd x\\
&-\frac{(k-j)}{\tau}\int_{\RR^d}n^j\nabla|\nabla c|^2\cdot\nabla \phi^R_{x_0}|\nabla c|^{2k-2j-2}\dd x\\
&-\frac{(k-j)(k-j-1)}{\tau}\int_{\RR^d}n^j|\nabla|\nabla c|^2|^2|\nabla c|^{2k-2j-4}\phi^R_{x_0}\dd x\\
&-\frac{2(k-j)}{\tau}\int_{\RR^d} n^j|D^2 c|^2|\nabla c|^{2k-2j-2}\phi^R_{x_0}\dd x,
\end{align*}
where by H\"{o}lder's inequality, we obtain
\begin{align*}
&-\frac{(k-j)j}{\tau}\int_{\RR^d}n^{j-1}\nabla n\cdot\nabla|\nabla c|^2|\nabla c|^{2k-2j-2}\phi^R_{x_0}\dd x\\
\le&\frac{j^2(k-j)^2}{2\tau^2}\int_{\RR^d}n^{j-1}|\nabla|\nabla c|^2|^2|\nabla c|^{2k-2j-2}\phi^R_{x_0}\dd x+\frac{1}{2}\int_{\RR^d}n^{j-1}|\nabla n|^2|\nabla c|^{2k-2j-2}\phi^R_{x_0}\dd x\\
=&\frac{j^2(k-j)^2}{2\tau^2}\int_{\RR^d}\left(n^j|\nabla c|^{2k-2j-4}\right)^{\frac{j-1}{j}}\left(|\nabla c|^{2k-4}\right)^{\frac{1}{j}}|\nabla |\nabla c|^2|^2\phi^R_{x_0}\dd x\\
&+\frac{1}{2}\int_{\RR^d}n^{j-1}|\nabla n|^2|\nabla c|^{2k-2j-2}\phi^R_{x_0}\dd x\\
\le&\frac{1}{2}\int_{\RR^d}n^{j-1}|\nabla n|^2|\nabla c|^{2k-2j-2}\phi^R_{x_0}\dd x+\frac{k-j}{8\tau}\int_{\RR^d}n^j|\nabla c|^{2k-2j-4}|\nabla |\nabla c|^2|^2\phi^R_{x_0}\dd x\\
&+\frac{8^{j-1}j^{2j}(k-j)^{j+1}}{\tau^{j+1}}\int_{\RR^d}|\nabla c|^{2k-4}|\nabla |\nabla c|^2|^2\phi^R_{x_0}\dd x.
\end{align*}
Noting the fact that
\begin{equation}\label{D^2}
|\nabla|\nabla c|^2|^2=|2D^2c\cdot\nabla c|^2\le 4|D^2 c|^2|\nabla c|^2,
\end{equation}
hence, from the above inequality we get that
\begin{align*}
&-\frac{(k-j)j}{\tau}\int_{\RR^d}n^{j-1}\nabla n\cdot\nabla|\nabla c|^2|\nabla c|^{2k-2j-2}\phi^R_{x_0}\dd x\\
\le&\frac{1}{2}\int_{\RR^d}n^{j-1}|\nabla n|^2|\nabla c|^{2k-2j-2}\phi^R_{x_0,}\dd x+\frac{k-j}{2\tau}\int_{\RR^d} n^j|D^2 c|^2|\nabla c|^{2k-2j-2}\phi^R_{x_0}\dd x\\
&+\frac{8^{j-1}j^{2j}(k-j)^{j+1}}{\tau^{j+1}}\int_{\RR^d}|\nabla c|^{2k-4}|\nabla |\nabla c|^2|^2\phi^R_{x_0}\dd x.
\end{align*}
By H\"{o}lder's inequality and Young's inequality, it is easy to check that
\begin{align*}
&-\frac{k-j}{\tau}\int_{\RR^d}n^j\nabla|\nabla c|^2\cdot\nabla \phi^R_{x_0}|\nabla c|^{2k-2j-2}\dd x\\
\le&\frac{k-j}{4\tau}\int_{\RR^d}n^j|\nabla|\nabla c|^2|^2|\nabla c|^{2k-2j-4}\phi^R_{x_0}\dd x+\frac{k-j}{\tau}\int_{\RR^d}n^j|\nabla c|^{2k-2j}(\phi^R_{x_0})^{-1}|\nabla \phi^R_{x_0}|^2\dd x\\
\le&\frac{k-j}{\tau}\int_{\RR^d}n^j|D^2c|^2|\nabla c|^{2k-2j-2}\phi^R_{x_0}\dd x+\frac{k-j}{\tau}\int_{\RR^d}n^j|\nabla c|^{2k-2j}\phi^{-1}_{x_0,R}|\nabla \phi^R_{x_0}|^2\dd x.
\end{align*}
To sum up, $I_5$ can be bounded as
\begin{equation}\label{I5}
\begin{aligned}
I_5\le &
-\frac{(k-j)(k-j-1)}{\tau}\int_{\RR^d}n^j|\nabla|\nabla c|^2|^2|\nabla c|^{2k-2j-4}\phi^R_{x_0}\dd x\\
&-\frac{k-j}{2\tau}\int_{\RR^d} n^j|D^2 c|^2|\nabla c|^{2k-2j-2}\phi^R_{x_0}\dd x+\frac{1}{2}\int_{\RR^d}n^{j-1}|\nabla n|^2|\nabla c|^{2k-2j-2}\phi^R_{x_0}\dd x\\
&+\frac{8^{j-1}j^{2j}(k-j)^{j+1}}{\tau^{j+1}}\int_{\RR^d}|\nabla c|^{2k-4}|\nabla |\nabla c|^2|^2\phi^R_{x_0}\dd x\\&+\frac{k-j}{\tau}\int_{\RR^d}n^j|\nabla c|^{2k-2j}(\phi^R_{x_0})^{-1}|\nabla \phi^R_{x_0}|^2\dd x.
\end{aligned}
\end{equation}
For $I_6$, it is easy to check that
\begin{equation}\label{I6}
\begin{aligned}
I_6\le& \frac{1}{4}\int_{\RR^d}n^{j-1}|\nabla n|^2|\nabla c|^{2k-2j-2}\phi^R_{x_0}\dd x+\frac{4(k-j)^2}{\tau^2}\int_{\RR^d}n^{j+1}|\nabla c|^{2k-2j}\phi^R_{x_0}\dd x.
\end{aligned}
\end{equation}
Using integration by parts, from $I_1$ we infer that
\begin{equation}\nonumber
\begin{aligned}
I_1=&-j(j-1)\int_{\RR^d}n^{j-2}|\nabla n|^2|\nabla c|^{2k-2j}\phi^R_{x_0}\dd x-j\int_{\RR^d}n^{j-1}\nabla n\cdot\nabla\phi^R_{x_0} |\nabla c|^{2k-2j}\dd x\\
&-j(j-k)\int_{\RR^d}n^{j-1}\nabla n\cdot\nabla|\nabla c|^2|\nabla c|^{2k-2j-2}\phi^R_{x_0}\dd x
:=I_{1,1}+I_{1,2}+I_{1,3}.
\end{aligned}
\end{equation}
In terms of $I_{1,2}$, using H\"{o}lder's inequality we obtain that
\begin{align*}
I_{1,2}\le \frac{1}{8}\int_{\RR^d}n^{j-1}|\nabla n|^2|\nabla c|^{2k-2j-2}\phi^R_{x_0}\dd x+2j^2\int_{\RR^d}n^{j-1}|\nabla c|^{2k-2j+2}(\phi^R_{x_0})^{-1}|\nabla \phi^R_{x_0}|^2\dd x.
\end{align*}
Similarly, $I_{1,3}$ can be bounded as
\begin{align*}
I_{1,3}\le& 2j^2(j-k)^2\int_{\RR^d}n^{j-1}|\nabla|\nabla c|^2|^2|\nabla c|^{2k-2j+2}\phi^{R}_{x_0}\dd x \\&+\frac{1}{8}\int_{\RR^d}n^{j-1}|\nabla n|^2|\nabla c|^{2k-2j-2}\phi^{R}_{x_0}\dd x\\
=&2j^2(j-k)^2\int_{\RR^d}\left(n^j|\nabla c|^{2k-2j-4}\right)^{\frac{j-1}{j}}\left(|\nabla c|^{2k-4}\right)^{\frac{1}{j}}|\nabla |\nabla c|^2|^2\phi^{R}_{x_0}\dd x\\
&+\frac{1}{8}\int_{\RR^d}n^{j-1}|\nabla n|^2|\nabla c|^{2k-2j-2}\phi^{R}_{x_0}\dd x\\
\le&\frac{1}{8}\int_{\RR^d}n^{j-1}|\nabla n|^2|\nabla c|^{2k-2j-2}\phi^{R}_{x_0}\dd x
+\frac{k-j}{16 \tau}\int_{\RR^d}n^j|\nabla c|^{2k-2j-4}|\nabla |\nabla c|^2|^2\phi^{R}_{x_0}\dd x\\
&+\frac{(32\tau j^2)^j(k-j)^{j+1}}{16\tau}\int_{\RR^d}|\nabla c|^{2k-4}|\nabla |\nabla c|^2|^2\phi^{R}_{x_0}\dd x
\end{align*}
Taking advantage of \eqref{D^2}, we can get that
\begin{align*}
I_{1,3}\le&\frac{1}{8}\int_{\RR^d}n^{j-1}|\nabla n|^2|\nabla c|^{2k-2j-2}\phi^{R}_{x_0}\dd x
+\frac{k-j}{4\tau}\int_{\RR^d}n^j|\nabla c|^{2k-2j-2}|D^2 c|^2\phi^{R}_{x_0}\dd x\\
&+\frac{(32\tau j^2)^j(k-j)^{j+1}}{16\tau}\int_{\RR^d}|\nabla c|^{2k-4}|\nabla |\nabla c|^2|^2\phi^{R}_{x_0}\dd x.
\end{align*}
Therefore, $I_1$ can be bounded by
\begin{equation}\label{I1}
\begin{aligned}
I_1\le& -j(j-1)\int_{\RR^d}n^{j-2}|\nabla n|^2|\nabla c|^{2k-2j}\phi^{R}_{x_0}\dd x+\frac{1}{4}\int_{\RR^d}n^{j-1}|\nabla n|^2|\nabla c|^{2k-2j-2}\phi^{R}_{x_0}\dd x\\
&+2j^2\int_{\RR^d}n^{j-1}|\nabla c|^{2k-2j+2}(\phi^{R}_{x_0})^{-1}|\nabla \phi^{R}_{x_0}|^2\dd x+\frac{k-j}{4\tau}\int_{\RR^d}n^j|\nabla c|^{2k-2j-2}|D^2 c|^2\phi^{R}_{x_0}\dd x\\
&+\frac{(32\tau j^2)^j(k-j)^{j+1}}{16\tau}\int_{\RR^d}|\nabla c|^{2k-4}|\nabla |\nabla c|^2|^2\phi^{R}_{x_0}\dd x.
\end{aligned}
\end{equation}
By integration by parts, $I_2$ can be decomposed as the following three parts:
\begin{align*}
I_2=&\chi j\int_{\RR^d}n^j\nabla c\cdot\nabla |\nabla c|^{2k-2j} \phi^{R}_{x_0}\dd x+\chi j\int_{\RR^d}n\nabla c\cdot\nabla n^{j-1}|\nabla c|^{2k-2j}\phi^{R}_{x_0}\dd x\\
&+\chi j\int_{\RR^d}n^j\nabla c\cdot\nabla\phi^{R}_{x_0}|\nabla c|^{2k-2j}:=I_{2,1}+I_{2,2}+I_{2,3}.
\end{align*}
For $I_{2,1}$, by Young's inequality and \eqref{D^2}, we have \begin{align*}
I_{2,1}\le& \frac{k-j}{32\tau}\int_{\RR^d} n^j|\nabla|\nabla c|^2|^2|\nabla c|^{2k-2j-4}\phi^R_{x_0}\dd x+ 8\tau\chi^2j^2(k-j)\int_{\RR^d} n^j|\nabla c|^{2k-2j+2}\phi^R_{x_0}\dd x\\
\le&\frac{k-j}{8\tau}\int_{\RR^d} n^j|D^2 c|^2|\nabla c|^{2k-2j-2}\phi^R_{x_0}\dd x
+ 8\tau\chi^2j^2(k-j)\int_{\RR^d} n^j|\nabla c|^{2k-2j+2}\phi^R_{x_0}\dd x.
\end{align*}
Similarly, $I_{2,2}$ can be bounded as
\begin{align*}
I_{2,2}
\le&\frac{j(j-1)}{2}\int_{\RR^d}n^{j-2}|\nabla n|^2|\nabla c|^{2k-2j}\phi^R_{x_0}\dd x+\frac{\chi^2j(j-1)}{2}\int_{\RR^d} n^j|\nabla c|^{2k-2j+2}\phi^R_{x_0}\dd x,
\end{align*}
and for $I_{2,3}$, we deduce that
\begin{align*}
I_{2,3}\le& \frac{\chi^2 j}{2}\int_{\RR^d} n^j|\nabla c|^{2k-2j+2}\phi^R_{x_0}\dd x+\frac{ j}{2}\int_{\RR^d} n^j|\nabla c|^{2k-2j}(\phi^R_{x_0, R})^{-1}|\nabla \phi^R_{x_0}|^2\dd x\\
\le&\frac{\chi^2 j}{2}\int_{\RR^d} n^j|\nabla c|^{2k-2j+2}\phi^R_{x_0}\dd x+\frac{ j^2}{2k}\int_{\RR^d} n^k(\phi^R_{x_0})^{-1}|\nabla \phi^R_{x_0}|^2\dd x\\
&+\frac{ j(k-j)}{2k}\int_{\RR^d} |\nabla c|^{2k}(\phi^R_{x_0})^{-1}|\nabla \phi^R_{x_0}|^2\dd x.
\end{align*}
To sum up, $I_2$ can be bounded by the following inequality.
\begin{align*}
I_2\le &\frac{k-j}{8\tau}\int_{\RR^d} n^j|D^2 c|^2|\nabla c|^{2k-2j-2}\phi^R_{x_0}\dd x+\frac{j(j-1)}{2}\int_{\RR^d}n^{j-2}|\nabla n|^2|\nabla c|^{2k-2j}\phi^R_{x_0}\dd x\\
&+ \chi^2\left(8\tau j^2(k-j)+\frac{j(j-1)}{2}+\frac{j}{2}\right)\int_{\RR^d} n^j|\nabla c|^{2k-2j+2}\phi^R_{x_0}\dd x\\
&+\frac{ j^2}{2k}\int_{\RR^d} n^k(\phi^R_{x_0})^{-1}|\nabla \phi^R_{x_0}|^2\dd x
+\frac{ j(k-j)}{2k}\int_{\RR^d} |\nabla c|^{2k}(\phi^R_{x_0})^{-1}|\nabla \phi^R_{x_0}|^2\dd x.
\end{align*}
By H\"{o}lder's and Young's inequalities, we obtain that
\begin{align*}
&\int_{\RR^d} n^j|\nabla c|^{2k-2j+2}\phi^R_{x_0}\dd x\\
=&\int_{\RR^d}(n^2|\nabla c|^{2k-2})^{\frac{1}{j-1}}(n^{j+1}|\nabla c|^{2k-2j})^{\frac{j-2}{j-1}}\phi^R_{x_0}\dd x\\
\le&\frac{1}{j-1}\int_{\RR^d}n^2|\nabla c|^{2k-2}\phi^R_{x_0}\dd x
+\frac{j-2}{j-1}\int_{\RR^d}n^{j+1}|\nabla c|^{2k-2j}\phi^R_{x_0}\dd x.
\end{align*}
Therefore, we conclude that
\begin{equation}\label{I2re}
\begin{aligned}
I_2\le &\frac{k-j}{8\tau}\int_{\RR^d} n^j|D^2 c|^2|\nabla c|^{2k-2j-2}\phi^R_{x_0}\dd x+\frac{j(j-1)}{2}\int_{\RR^d}n^{j-2}|\nabla n|^2|\nabla c|^{2k-2j}\phi^R_{x_0}\dd x\\
&+\chi^2j\left(8\tau j(k-j)+\frac{j}{2}\right)\int_{\RR^d}n^2|\nabla c|^{2k-2}\phi^R_{x_0}\dd x
+\frac{ j^2}{2k}\int_{\RR^d} n^k\left(\phi^R_{x_0}\right)^{-1}|\nabla \phi^R_{x_0}|^2\dd x\\
&+\chi^2j\left(8\tau j(k-j)+\frac{j}{2}\right)\int_{\RR^d}n^{j+1}|\nabla c|^{2k-2j}\phi^R_{x_0}\dd x\\
&
+\frac{ j(k-j)}{2k}\int_{\RR^d} |\nabla c|^{2k}(\phi^R_{x_0})^{-1}|\nabla \phi^R_{x_0}|^2\dd x.
\end{aligned}
\end{equation}
To sum up, combining with inequalities \eqref{I5}--\eqref{I2re} we can get that
\begin{align*}
&\frac{\dd}{\dd t}\int_{\RR^d}n^j|\nabla c|^{2k-2j}\phi^R_{x_0}\dd x+\frac{2(k-j)}{\tau}\int_{\RR^d}n^j|\nabla c|^{2k-2j}\phi^R_{x_0}\dd x\\
&+\frac{k-j}{8\tau}\int_{\RR^d} n^j|D^2 c|^2|\nabla c|^{2k-2j-2}\phi^R_{x_0}\dd x+\frac{j(j-1)}{2}\int_{\RR^d}n^{j-2}|\nabla n|^2|\nabla c|^{2k-2j}\phi^R_{x_0}\dd x\\
&+\frac{(k-j)(k-j-1)}{\tau}\int_{\RR^d}n^j|\nabla|\nabla c|^2|^2|\nabla c|^{2k-2j-4}\phi^R_{x_0}\dd x\\
\le&\int_{\RR^d}n^{j-1}|\nabla n|^2|\nabla c|^{2k-2j-2}\phi^R_{x_0}\dd x\\
&+(k-j)^{j+1}\left(\frac{2^{2j-3}j^{2j}}{\tau^{j+1}}+\frac{(32\tau j^2)^j}{16\tau}\right)\int_{\RR^d}|\nabla c|^{2k-4}|\nabla |\nabla c|^2|^2\phi^R_{x_0}\dd x\\
&+\frac{k-j}{\tau}\int_{\RR^d}n^j|\nabla c|^{2k-2j}\phi^{-1}_{x_0,R}|\nabla \phi^R_{x_0}|^2\dd x+2j^2\int_{\RR^d}n^{j-1}|\nabla c|^{2k-2j+2}\phi^{-1}_{x_0,R}|\nabla \phi^R_{x_0}|^2\dd x\\
&+\left(\chi^2j\left(8\tau j(k-j)+\frac{j}{2}\right)+\frac{4(k-j)^2}{\tau^2}-\mu j\right)\int_{\RR^d}n^{j+1}|\nabla c|^{2k-2j}\phi^R_{x_0}\dd x\\
&+\chi^2j\left(8\tau j(k-j)+\frac{j}{2}\right)\int_{\RR^d}n^2|\nabla c|^{2k-2}\phi^R_{x_0}\dd x+\frac{ j^2}{2k}\int_{\RR^d} n^k(\phi^R_{x_0})^{-1}|\nabla \phi^R_{x_0}|^2\dd x\\
&+\frac{ j(k-j)}{2k}\int_{\RR^d} |\nabla c|^{2k}(\phi^R_{x_0})^{-1}|\nabla \phi^R_{x_0}|^2\dd x+\lambda j\int_{\RR^d}n^j|\nabla c|^{2k-2j}\phi^R_{x_0}\dd x.
\end{align*}
Taking advantage of H\"{o}lder's and Young's inequalities, we get that
\begin{align*}
&\int_{\RR^d}n^j|\nabla c|^{2k-2j}(\phi^R_{x_0})^{-1}|\nabla \phi^R_{x_0}|^2\dd x\\
\le&\int_{\RR^d}n^k(\phi^R_{x_0})^{-1}|\nabla \phi^R_{x_0}|^2\dd x+
\int_{\RR^d}|\nabla c|^{2k}(\phi^R_{x_0})^{-1}|\nabla \phi^R_{x_0}|^2\dd x\\
\le&\frac{C3^d}{R^2}\|n\|^k_{L^k_{\uloc}(\RR^d)}+\frac{C3^d}{R^2}\|\nabla c\|^{2k}_{L^{2k}_{\uloc}(\RR^d)},
\end{align*}
\begin{align*}
\int_{\RR^d}n^{j-1}|\nabla c|^{2k-2j+2}(\phi^R_{x_0})^{-1}|\nabla \phi^R_{x_0}|^2\dd x
\le&\frac{C3^d}{R^2}\|n\|^k_{L^k_{\uloc}(\RR^d)}+\frac{C3^d}{R^2}\|\nabla c\|^{2k}_{L^{2k}_{\uloc}(\RR^d)},
\end{align*}
and
\begin{align*}
\int_{\RR^d}n^j|\nabla c|^{2k-2j}\phi^R_{x_0}\dd x\le&\int_{\RR^d}n^{j+1}|\nabla c|^{2k-2j}\phi^R_{x_0}\dd x+\int_{\RR^d}|\nabla c|^{2k-2j}\phi^R_{x_0}\dd x\\
\le&\int_{\RR^d}n^{j+1}|\nabla c|^{2k-2j}\phi^R_{x_0}\dd x+\int_{\RR^d}|\nabla c|^{2k-2}\phi^R_{x_0}\dd x+CR^{\frac{d(k-1)}{j-1}}.
\end{align*}
Hence, we can further obtain that
\begin{align*}
&\frac{\dd}{\dd t}\int_{\RR^d}n^j|\nabla c|^{2k-2j}\phi_{x_0,R}\dd x+\frac{j(j-1)}{2}\int_{\RR^d}n^{j-2}|\nabla n|^2|\nabla c|^{2k-2j}\phi_{x_0,R}\dd x\\
\le&\int_{\RR^d}n^{j-1}|\nabla n|^2|\nabla c|^{2k-2j-2}\phi^R_{x_0}\dd x\\&+(k-j)^{j+1}\left(\frac{2^{2j-3}j^{2j}}{\tau^{j+1}}+\frac{(32\tau j^2)^j}{16\tau}\right)\int_{\RR^d}|\nabla c|^{2k-4}|\nabla |\nabla c|^2|^2\phi^R_{x_0}\dd x\\
&+\left(\chi^2j\left(8\tau j(k-j)+\frac{j}{2}\right)+\frac{4(k-j)^2}{\tau^2}+\lambda j-\mu j\right)\int_{\RR^d}n^{j+1}|\nabla c|^{2k-2j}\phi_{x_0,R}\dd x\\
&+C\lambda jR^{\frac{d(k-1)}{j-1}}+\chi^2j\left(8\tau j(k-j)+\frac{j}{2}\right)\int_{\RR^d}n^2|\nabla c|^{2k-2}\phi^R_{x_0}\dd x+\lambda j\int_{\RR^d}|\nabla c|^{2k-2}\phi^R_{x_0}\dd x\\
&+\frac{C3^d}{R^2}\left(\frac{j^2}{2k}+\frac{k-j}{\tau }+2j^2\right)\left(\|n\|^k_{L^k_\text{\uloc}(\RR^d)}+\|\nabla c\|^{2k}_{L^{2k}_\text{\uloc}(\RR^d)}\right).
\end{align*}
This implies the desired result.
\end{proof}
\begin{proposition}\label{sumj}Assume $\tau, \chi,\mu>0$, $\lambda\ge 0$ and $R\ge1$. Let $k\in\NN$ with $k\ge 3$. Then there exist constant $b_1,\cdots, b_k$, $\mu_0=\mu_0(\tau,  \chi, \lambda, k)$ and an absolute constant $C$ such that, for $\mu\ge\mu_0$, the solution $(n, c)$ of equations \eqref{che-M}-\eqref{che-I-M} fulfills
\begin{align*}
&\frac{\dd}{\dd t}\Big(\int_{\RR^d}|\nabla c|^{2k}\phi^R_{x_0}\dd x+\sum_{j=1}^k b_j\int_{\RR^d}n^j|\nabla c|^{2k-2j}\phi^R_{x_0}\dd x\Big)\\
&+\frac{k(k-1)}{16\tau}\int_{\RR^d}|\nabla c|^{2k-4}|\nabla|\nabla c|^2|^2\phi^R_{x_0}\dd x+\frac{b_k}{8}\int_{\RR^d}n^{k-2}|\nabla n|^2\phi^R_{x_0}\dd x\\
\le&\frac{C(\lambda+1)}{\tau}\int_{\RR^d}|\nabla c|^{2k-2}\phi^R_{x_0}\dd x
+\left(\frac{C3^d}{\tau^2 R^{2}}+\frac{k(k-1)}{4\tau R}+\frac{C3^d}{\tau R^2}\right)\|n\|^k_{L^k_\text{\uloc}(\RR^d)}\\
&+\left(\frac{C3^d k}{\tau R^2}+\frac{C3^d(1+\frac{1}{\tau})}{ \tau R^{2}}+\frac{k(k-1)}{4\tau R}+\frac{C3^d}{\tau R^{2k}}\right)\|\nabla c\|^{2k}_{L^{2k}_{\uloc}(\RR^d)}+\frac{CR^d}{\tau}+C(\lambda+1) R^d.
\end{align*}
\end{proposition}
\begin{proof}
By performing Proposition \ref{lemmank}-Proposition \ref{njc2k-2j}, it is easy to check by an easy calculation  that
\begin{align*}
&\frac{\dd}{\dd t}\Bigg(\int_{\RR^d}|\nabla c|^{2k}\phi^R_{x_0}\dd x+\sum_{j=1}^k b_j\int_{\RR^d}n^j|\nabla c|^{2k-2j}\phi^R_{x_0}\dd x\Bigg)\\
&+\frac{k(k-1)}{16\tau}\int_{\RR^d}|\nabla c|^{2k-4}|\nabla|\nabla c|^2|^2\phi^R_{x_0}\dd x+\frac{b_k}{8}\int_{\RR^d}n^{k-2}|\nabla n|^2\phi^R_{x_0}\dd x\\
\le&-\frac{k(k-1)}{4\tau}\int_{\RR^d}|\nabla |\nabla c|^2|^2|\nabla c|^{2k-4}\phi^R_{x_0}\dd x
+\frac{(d+1)k}{\tau}\int_{\RR^d} n^2|\nabla c|^{2k-2}\phi^R_{x_0}\dd x\\&+\frac{C3^dk}{\tau R^2}\|\nabla c\|^{2k}_{L^{2k}_{\uloc}(\RR^d)}+b_1C_1\int_{\RR^d}|\nabla|\nabla c|^2|^2|\nabla c|^{2k-4}\phi^R_{x_0}\dd x
+\frac{\lambda b_1}{2}\int_{\RR^d}|\nabla c|^{2k-2}\phi^R_{x_0}\dd x\\
&+\frac{Cb_13^d(1+\frac{1}{\tau})}{ R^{2}}\|\nabla c\|^{2k}_{L^{2k}_{\uloc}(\RR^d)}+\frac{Cb_13^d}{\tau R^{2}}\|n\|^{k}_{L^{k}_{\uloc}(\RR^d)}+(C_1-\mu)b_1
\int_{\RR^d}n^2|\nabla c|^{2k-2}\phi^R_{x_0}\dd x\\
&+b_1\int_{\RR^d}|\nabla n|^2|\nabla c|^{2k-4}\phi^R_{x_0}\dd x+\sum_{j=2}^{k-1}b_j\Bigg(-\frac{j(j-1)}{4}\int_{\RR^d}n^{j-2}|\nabla n|^2|\nabla c|^{2k-2j}\phi^R_{x_0}\dd x\\
&+\int_{\RR^d}n^{j-1}|\nabla n|^2|\nabla c|^{2k-2j-2}\phi^R_{x_0}\dd x+C_j\int_{\RR^d}|\nabla c|^{2k-4}|\nabla |\nabla c|^2|^2\phi^R_{x_0}\dd x\\
&+(C_j-\mu j)\int_{\RR^d}n^{j+1}|\nabla c|^{2k-2j}\phi^R_{x_0}\dd x+CR^dj+C_j\int_{\RR^d}n^2|\nabla c|^{2k-2}\phi^R_{x_0}\dd x\\
&+\frac{C_j}{R}\|n\|^k_{L^k_\text{\uloc}(\RR^d)}
+\lambda j\int_{\RR^d}|\nabla c|^{2k-2}\phi^R_{x_0}\dd x+\frac{C_j}{R}\|\nabla c\|^{2k}_{L^{2k}_{\uloc}(\RR^d)}\Bigg)\\
&+b_k\Bigg(-\frac{k(k-1)}{4}\int_{\RR^d}|\nabla n|^2 n^{k-2}\phi^R_{x_0}\dd x+\frac{C3^dk}{2(k-1)R^2}\|n\|^k_{L^k_{\uloc}(\RR^d)}+\frac{C3^dk}{R^{2k}}\|\nabla c\|^{2k}_{L^{2k}_{\uloc}(\RR^d)}\\
&+k\int_{\RR^d}n^{2}|\nabla c|^{2k-2}\phi^R_{x_0}\dd x+(C_k-\mu k)\int_{\RR^d} n^{k+1}\phi^R_{x_0}\dd x+C(\lambda+1) R^dk\Bigg)\\
&+\frac{k(k-1)}{16\tau}\int_{\RR^d}|\nabla c|^{2k-4}|\nabla|\nabla c|^2|^2\phi^R_{x_0}\dd x+\frac{b_k}{8}\int_{\RR^d}n^{k-2}|\nabla n|^2\phi^R_{x_0}\dd x.
\end{align*}
By rearranging the above inequality, we can obtain that
\begin{align*}
&\frac{\dd}{\dd t}\Big(\int_{\RR^d}|\nabla c|^{2k}\phi^R_{x_0}\dd x+\sum_{j=1}^k b_j\int_{\RR^d}n^j|\nabla c|^{2k-2j}\phi^R_{x_0}\dd x\Big)\\
\le&\left(-\frac{k(k-1)}{4\tau}+\sum_{j=1}^{k-1}b_jC_j+\frac{k(k-1)}{16\tau}\right)\int_{\RR^d}|\nabla c|^{2k-4}|\nabla |\nabla c|^2|^2\phi^R_{x_0}\dd x\\
&+\left(\frac{(d+1)k}{\tau}+(C_1-\mu)b_1+\sum_{j=2}^{k-1}b_j C_j+kb_k\right)\int_{\RR^d} n^2|\nabla c|^{2k-2}\phi^R_{x_0}\dd x\\
&+\left(\sum_{j=2}^k\left(b_{j-1}-\frac{j(j-1)b_j}{4}\right)+\frac{b_k}{8}\right)\int_{\RR^d}n^{j-2}|\nabla n|^2|\nabla c|^{2k-2j}\phi^R_{x_0}\dd x\\
&+\left(\frac{b_1\lambda}{2}+\lambda\sum_{j=2}^{k-1}jb_j\right)\int_{\RR^d}|\nabla c|^{2k-2}\phi^R_{x_0}\dd x
\\&+\left(\frac{Cb_13^d}{\tau R^{2}}+\sum_{j=2}^{k-1}\frac{b_jC_j}{R^2}+\frac{C3^dkb_k}{2(k-1)R^2}\right)\|n\|^k_{L^k_\text{uloc}(\RR^d)}\\
&+\left(\frac{C3^dk}{\tau R^2}+\frac{Cb_13^d(1+\frac{1}{\tau})}{ R^{2}}+\sum_{j=2}^{k-1}\frac{b_jC_j}{R^2}+\frac{C3^dkb_k}{R^{2k}}\right)\|\nabla c\|^{2k}_{L^{2k}_{\uloc}(\RR^d)}\\
&+\sum_{j=2}^kb_j(C_j-\mu j)\int_{\RR^d}n^{j+1}|\nabla c|^{2k-2j}\phi^R_{x_0}\dd x+CR^d\sum_{j=2}^{k-1}jb_j+C(\lambda+1) R^dkb_k.
\end{align*}
It is easy to check that, for any fixed $\tau$, $\lambda$ and $\chi$, there exists a positive constant $C_0=C_0(\tau, \lambda, \chi)$ such that $$C_j\le C_0k^{3k+1},\quad(j=1,\cdots, k-1).$$
 Then $\forall\, j=1,2,\cdots, k$, we set $$b_j= \frac{k^{2-5k}(k-1)}{16\tau C_0}k^{2j}.$$ By this definition, it yields that
\begin{equation}\label{sumbc}
\sum_{j=1}^{k-1}b_jC_j\le \frac{k^{3-2k}(k-1)}{16\tau}\sum_{j=1}^{k-1}k^{2j}<\frac{k(k-1)}{8\tau}
\end{equation}
and
\begin{align*}
\frac{b_{j-1}}{b_j}=k^{-2},\qquad j=2,\cdots,k.
\end{align*}
Therefore, it yields that
\begin{equation}\label{con1}
-\frac{k(k-1)}{4\tau}+\sum_{j=1}^{k-1}b_jC_j+\frac{k(k-1)}{16\tau}<0
\end{equation}
and
\begin{equation}\label{con2}
\sum_{j=2}^k\left(b_{j-1}-\frac{j(j-1)b_j}{4}\right)+\frac{b_k}{8}\le b_k\sum_{j=2}^k\left(k^{-2}-\frac{j(j-1)}{4}+\frac{1}{8}\right)<0.
\end{equation}
We choose some $\mu_0\ge C_0k^{3k+1}\ge C_j$ such that
\begin{equation}\label{con3}
C_j-\mu_0j\le 0,\qquad j=2,3,\cdots,k
\end{equation}
and by \eqref{sumbc},
\begin{equation}\label{con4}
\begin{split}
&\frac{(d+1)k}{\tau}+(C_1-\mu_0)b_1+\sum_{j=2}^{k-1}b_j C_j+kb_k\\<&\frac{(d+1)k}{\tau}+\frac{k(k-1)}{8\tau}+\frac{1}{16\tau C_0}-\mu_0 b_1<0,
\end{split}
\end{equation}
then taking $\mu\ge \mu_0$, from conditions \eqref{con1}--\eqref{con4} we infer that
\begin{align*}
&\frac{\dd}{\dd t}\Big(\int_{\RR^d}|\nabla c|^{2k}\phi^R_{x_0}\dd x+\sum_{j=1}^k b_j\int_{\RR^d}n^j|\nabla c|^{2k-2j}\phi^R_{x_0}\dd x\Big)\\
\le&\left(\frac{\lambda b_1}{2}+\frac{\lambda}{8\tau C_0}\right)\int_{\RR^d}|\nabla c|^{2k-2}\phi^R_{x_0}\dd x
+\left(\frac{C3^d}{\tau^2 R^{2}}+\frac{k(k-1)}{8\tau R^2}+\frac{C3^d}{\tau C_0R^2}\right)\|n\|^k_{L^k_\text{\uloc}(\RR^d)}\\
&+\left(\frac{C3^d k}{\tau R^2}+\frac{C3^d(1+\frac{1}{\tau})}{ \tau R^{2}}+\frac{k(k-1)}{8\tau R^2}+\frac{C3^d}{\tau C_0 R^{2k}}\right)\|\nabla c\|^{2k}_{L^{2k}_{\uloc}(\RR^d)}\\&+\frac{CR^d}{\tau C_0}+C(\lambda+1) R^d.
\end{align*}
This implies the required estimate.
\end{proof}
Before establishing the uniformly local $L^p(\RR^d)$ estimate for $n$, we firstly present a lemma in terms of the generalized interpolation inequality needed in the sequel.
\begin{lemma}\label{Poincare} Assume $u\in H^{1}(\RR^d)$  satisfying 
$\displaystyle\int_{\RR^d}|u|^{\frac{2}{k}}(x)\dd x<\infty $ for $k\geq2.$
Then there exists a constant $C$, depending on $d$ such that
\begin{align*}
\|u\|^2_{L^2(\RR^d)}\le C\|\nabla u\|^2_{L^2(\RR^d)}+C^k\Big(\int_{\RR^d}|u|^{\frac{2}{k}}\dd x\Big)^k.
\end{align*}
\end{lemma}
\begin{proof}By the interpolation theorem and the Young inequality, it yields that
\begin{equation}\label{Poin}
\begin{aligned}
\|u\|^2_{L^2(\RR^d)}\le& C\|\nabla u\|^2_{L^2(\RR^d)}+C\|u\|^2_{L^1(\RR^d)}.
\end{aligned}
\end{equation}
By H\"{o}lder's inequality, it yields that
\begin{align*}
C\left(\int_{\RR^d}|u| \dd x\right)^2=&C\left(\int_{\RR^d}|u|^{\frac{1}{k-1}}|u|^{\frac{k-2}{k-1}} \dd x\right)^2\\
\le &C\left(\int_{\RR^d}|u|^{\frac{2}{k}}\dd x\right)^{\frac{2k}{2k-2}}\left(\int_{\RR^d}|u|^2\dd x\right)^{\frac{k-2}{k-1}}\\
\le&\frac{1}{2}\int_{\RR^d}|u|^{2}\dd x+C^k\left(\int_{{\RR^d}}|u|^{\frac{2}{k}}\dd x\right)^{k}.
\end{align*}
Substituting the above estimate into inequality \eqref{Poin}, we conclude the desired result.
\end{proof}
\begin{proposition}\label{nkc2k}Let $\tau, \,\chi>0$ and $\lambda\ge 0$. Assume $k\in\NN$ with $k\ge 3$ and $\mu\ge\mu_0(\tau,\,\chi,\, k,\, \lambda)$ given in Proposition \ref{sumj}. Then there exist some $R>1$ and $C(\tau,\chi, k, \lambda, d,  \mu)$ such that the solution of problem \eqref{che-M}-\eqref{che-I-M} satisfies, for all $t\in (0, T_{\max})$,
\begin{align*}
&\|n\|^{k}_{L^\infty_t L^{k}_{\text{\uloc}}(\RR^d)}+\|\nabla c\|^{2k}_{L^\infty_tL^{2k}_{\text{\uloc}}(\RR^d)}\\\le &C(\lambda, \tau, \mu, \chi, d)\left(1+\|n_0\|^2_{L^1_{\text{\uloc}}(\RR^d)}+\|\nabla c_0\|^2_{L^2_{\text{\uloc}}(\RR^d)}+\|\nabla c_0\|^{2k}_{L^{2k}_{\text{\uloc}}(\RR^d)}+\|n_0\|^{k}_{L^{k}_{\text{\uloc}}(\RR^d)}\right).
\end{align*}
\end{proposition}
\begin{proof}
By virtue of Lemma \ref{Poincare}, we obtain that
\begin{equation*}
\begin{aligned}
&\int_{\RR^d}n^k\phi^R_{x_0}\dd x=\left\|n^{\frac{k}{2}}\left(\phi^R_{x_0}\right)^{\frac{1}{2}}\right\|^2_{L^2(\RR^d)}\\
\le& C\int_{\RR^d}\left|\nabla\left(n^{\frac{k}{2}}(\phi^R_{x_0})^{\frac{1}{2}}\right)\right|^2\dd x+C^k \left(\int_{B_{2R}(x_0)}n (\phi^R_{x_0})^{\frac{1}{k}}\dd x\right)^2\\
\le&Ck^2\int_{\RR^d}n^{k-2}|\nabla n|^2\phi^R_{x_0}\dd x+ C\int_{\RR^d}n^k |\nabla(\phi^R_{x_0})^{\frac{1}{2}}|^2\dd x+C^k 3^d\|n\|^2_{L^1_{\text{\uloc}}(\RR^d)}\\
\le&Ck^2\int_{\RR^d}n^{k-2}|\nabla n|^2\phi^R_{x_0}\dd x+ \frac{C}{R^2}\|n\|^k_{L^k_{\text{\uloc}}(\RR^d)}+C^k 3^d\|n\|^2_{L^1_{\text{\uloc}}(\RR^d)}.
\end{aligned}
\end{equation*}
Similarly, it is easy to check that
\begin{equation*}\nonumber
\begin{aligned}
&\int_{\RR^d}|\nabla c|^{2k}\phi^R_{x_0}\dd x\\\le&Ck^2\int_{\RR^d}|\nabla c|^{2k-4}|\nabla |\nabla c|^2|^2\phi^R_{x_0}\dd x+ \frac{C}{R^2}\|\nabla c\|^{2k}_{L^{2k}_{\text{\uloc}}(\RR^d)}+C^k 3^d\|\nabla c\|^2_{L^2_{\text{\uloc}}(\RR^d)}.
\end{aligned}
\end{equation*}
The above two inequalities imply that
\begin{equation*}\label{poincare}
\begin{aligned}
&\frac{k(k-1)}{16\tau}\int_{\RR^d}n^{k-2}|\nabla n|^2\phi^R_{x_0}\dd x+\frac{b_k}{8}\int_{\RR^d}|\nabla c|^{2k-4}|\nabla |\nabla c|^2|^2\phi^R_{x_0}\dd x\\
\ge&\frac{k(k-1)}{16\tau}\Big(\frac{C}{k^2}\int_{\RR^d}n^k\phi^R_{x_0}\dd x-\frac{C}{k^2R^2}\|n\|^k_{L^k_{\uloc}}
-\frac{C^{k}3^d}{k^2}\|n\|^2_{L^1_{\text{\uloc}}}\Big)\\
&+\frac{b_k}{8}\left(\frac{C}{k^2}\int_{\RR^d}|\nabla c|^{2k}\phi^R_{x_0}\dd x-\frac{C}{k^2R^2}\|\nabla c\|^{2k}_{L^{2k}_{\uloc}(\RR^d)}
-\frac{C^{k}3^d}{k^2}\|\nabla c\|^2_{L^2_{\text{\uloc}}(\RR^d)}\right).
\end{aligned}
\end{equation*}
We set
\[y(t):=\int_{\RR^d}|\nabla c|^{2k}\phi^R_{x_0}\dd x+\sum_{j=1}^k b_j\int_{\RR^d}n^j|\nabla c|^{2k-2j}\phi^R_{x_0}\dd x.\]
Taking advantage of Lemma \ref{sumj} and \eqref{poincare}, there exists an constant $C_1$ only depending on $d$ such that
\begin{align*}
&y'(t)+\frac{C_1k(k-1)}{16\tau k^2 }\int_{\RR^d}n^k\phi^R_{x_0}\dd x+\frac{C_1b_k}{ k^2}\int_{\RR^d}|\nabla c|^{2k}\phi^R_{x_0}\dd x\\
\le &\frac{3^dC^{k}k(k-1)}{\tau k^2}\|n\|^2_{L^1_{\text{\uloc}}(\RR^d)}
+\frac{C^kb_k3^d}{ k^2 }\|\nabla c\|^2_{L^2_{\text{\uloc}}(\RR^d)}+\frac{C(\lambda+1)}{\tau}\int_{\RR^d}|\nabla c|^{2k-2}\phi^R_{x_0}\dd x\\
&+\left(\frac{Ck(k-1)}{\tau k^2R^2}+\frac{C3^d}{\tau^2 R^{2}}+\frac{k(k-1)}{4\tau R^2}+\frac{C3^d}{\tau R^2}\right)\|n\|^k_{L^k_\text{\uloc}(\RR^d)}\\
&+\left(\frac{Cb_k}{\tau k^2 R^2}+\frac{C3^d k}{\tau R^2}+\frac{C3^d(1+\frac{1}{\tau})}{ \tau R^{2}}+\frac{k(k-1)}{4\tau R^2}+\frac{C3^d}{\tau R^{2k}}\right)\|\nabla c\|^{2k}_{L^{2k}_{\uloc}(\RR^d)}\\&+\frac{CR^d}{\tau}+C(\lambda+1) R^d.
\end{align*}
By H\"older's inequality, we have
\[\frac{C(\lambda+1)}{\tau}\int_{\RR^d}|\nabla c|^{2k-2}\phi^R_{x_0}\dd x\le\frac{C_1b_k}{2 k^2}\int_{\RR^d}|\nabla c|^{2k}\phi^R_{x_0}\dd x+C(\lambda, \tau, k, \chi)R^{d}.\]
For any fixed $k$, taking
$$c_0=\min\left\{\frac{C_1k(k-1)}{16\tau k^2 }, \frac{C_1b_k}{ k^2}\right\},$$
combining with the fact that there exists a constant $C_2=C_2(\chi,\lambda, k,\tau)$ such that
\begin{align*}
y(t)\le{C_2} \left(\int_{\RR^d}n^k\phi^R_{x_0}\dd x+\int_{\RR^d}|\nabla c|^{2k}\phi^R_{x_0}\dd x\right),
\end{align*}
we can get that
\begin{align*}
y'(t)\le& -\frac{c_0}{C_2} y(t)+\frac{3^dC^{k}k(k-1)}{\tau k^2}\|n\|^2_{L^1_{\text{uloc}}(\RR^d)}
+\frac{C^kb_k3^d}{ k^2 }\|\nabla c\|^2_{L^2_{\text{\uloc}}(\RR^d)}\\
&+\left(\frac{Ck(k-1)}{\tau k^2R^2}+\frac{C3^d}{\tau^2 R^{2}}+\frac{k(k-1)}{4\tau R^2}+\frac{C3^d}{\tau R^2}\right)\|n\|^k_{L^k_\text{\uloc}(\RR^d)}\\
&+\left(\frac{Cb_k}{\tau k^2 R^2}+\frac{C3^d k}{\tau R^2}+\frac{C3^d(1+\frac{1}{\tau})}{ \tau R^{2}}+\frac{k(k-1)}{4\tau R^2}+\frac{C3^d}{\tau R^{2k}}\right)\|\nabla c\|^{2k}_{L^{2k}_{\uloc}(\RR^d)}\\&+C(\lambda, \tau, k, \chi)R^{d}.
\end{align*}
Using Lemma \ref{n1c2} and the above inequality, we can infer that, for $R\ge R_0$,
\begin{align*}
y(t)\le \max\Bigg\{y(0),\,\,&
\frac{\widetilde{C}(d, \tau, k,\lambda, \chi)}{ R^2}\left(\|n\|^k_{L^k_\text{\uloc}(\RR^d)} +\|\nabla c\|^{2k}_{L^{2k}_\text{\uloc}(\RR^d)}\right)\\&+C(\lambda, \tau, \mu, \chi, d, R,k)\left(1+\|n_0\|^2_{L^1_{\text{\uloc}}(\RR^d)}+\|\nabla c_0\|^2_{L^2_{\text{\uloc}}(\RR^d)}\right)\Bigg\}.
\end{align*}
Now we choose some $R\ge R_0$ ($R_0$ is defined in Lemma \ref{n1c2}) satisfying that
$$\frac{\widetilde{C}(d, \tau, k,\lambda, \chi)}{ R^2}\le \min\left\{\frac{1}{2},\,\, \frac{b_k}{2}\right\}.$$
Following choice of $R$, one can check that $R$ depends on $d, \tau, k,\lambda, \chi$. Therefore, we have
\begin{align*}
&\frac{1}{2}\|\nabla c(t)\|^{2k}_{L^{2k}_{\text{\uloc}}(\RR^d)}+\frac{b_k}{2}\|n(t)\|^{k}_{L^{k}_{\text{\uloc}}(\RR^d)}\\
\le&y(0)+C(\lambda, \tau, \mu, \chi, d, R)\left(1+\|n_0\|^2_{L^1_{\text{\uloc}}(\RR^d)}+\|\nabla c_0\|^2_{L^2_{\text{\uloc}}(\RR^d)}\right)\\
\le&C(\lambda, \tau, \mu, \chi, d, k)\left(1+\|n_0\|^2_{L^1_{\text{\uloc}}(\RR^d)}+\|\nabla c_0\|^2_{L^2_{\text{\uloc}}(\RR^d)}+\|\nabla c_0\|^{2k}_{L^{2k}_{\text{\uloc}}(\RR^d)}+\|n_0\|^{k}_{L^{k}_{\text{\uloc}}(\RR^d)}\right),
\end{align*}
and therefore we complete the proof.
\end{proof}
\begin{proposition}\label{prop-bound}
Let $\tau, \,\chi>0$ and $\lambda\ge 0$. Assume $k\in\NN$ with $k\ge 3$ and $\mu\ge\mu_0(\tau,\,\chi,\, k,\, \lambda)$ given in Proposition \ref{sumj}. Then there exist some $R>1$ and $C(\tau,\chi, k, \lambda, d,  \mu)$ such that the solution of problem \eqref{che-M}-\eqref{che-I-M} satisfies
$$\|n(t)\|_{L^\infty(\RR^d)}+\|  c(t)\|_{W^{1,\infty}(\RR^d)}\le C \qquad \forall\, t\in(0, T_{\max}).$$
\end{proposition}
To show this proposition, we need the following lemma which can viewed as the generalized Young inequality.
\begin{lemma}\label{GYong}
Let $\varphi\in\mathcal{D}(\RR^d)$ and $f\in L^p_{\uloc}(\RR^d)$ with $1\leq p<\infty.$
Then there exists a constant $C$ such that
\begin{equation}\label{eq-local-1}
 \|\varphi * f\|_{L^{\infty}\left(\mathbb{R}^{d}\right)} \leq C\|f\|_{1,1}.
\end{equation}
In particular, we have that for $j\geq0,$
\begin{equation}\label{GY-lmf}
\|\dot{\Delta}_j f\|_{L^{\infty}\left(\mathbb{R}^{d}\right)}+\|\dot{S}_j f\|_{L^{\infty}\left(\mathbb{R}^{d}\right)}  \leq C2^{\frac{d}{p}j}\|f\|_{p,1}.
\end{equation}
 Here  and what in follows, we denoted by $\dot{\Delta}_{j}$ the homogeneous dyadic blocks and $\dot{S}_j$ the homogeneous low-frequency cut-off function respectively, see for example \cite[Chapter 2]{BCD}.
\end{lemma}
\begin{proof}
Letting $\mathcal{C}(k,k+1, x_0):=B_{2^{k+1}}(x_0)\backslash B_{2^{k}}(x_0),$ we  rewrite
\begin{align*}
(\varphi * f)(x_0)=&\int_{\RR^d}\varphi(x_0-y)f(y)\dd y\\
=&\int_{B_1(x_0)}\varphi(x_0-y)f(y)\dd y+\sum_{k\geq0}\int_{\mathcal{C}(k,k+1, x_0)}\varphi(x_0-y)f(y)\dd y.
\end{align*}
Since  $\varphi\in\mathcal{D}(\RR^d)$, we have by the H\"older inequality that
\begin{align*}
|(\varphi * f)(x_0)|\leq&C\|f\|_{1,1}+\sup_{x\in\RR^d}\big||x|^{2d}\varphi(x)\big|\sum_{k\geq0}\int_{\mathcal{C}(k,k+1, x_0)}|x_0-y|^{-{2d}}|f|(y)\dd y \\
\leq&C\|f\|_{1,1}+C\sum_{k\geq0}2^{-2dk}\int_{\mathcal{C}(k,k+1,x_0)} |f|(y)\dd y \\
\leq&C\|f\|_{1,1}+C\|f\|_{1,1}\sum_{k\geq0}2^{-2dk} 2^{kd},
\end{align*}
taking $\sup$ in terms of $x_0\in\RR^d$, which implies the desired estimate \eqref{eq-local-1}.

Next, we can show the second estimate by the scaling analysis. Let $f_{j}(x):=f\left(x / 2^{j}\right)$, we see that
$$
\big|\dot{\Delta}_{j} f(x)\big|=\left|\int_{\mathbb{R}^{d}} \varphi(y) f_{j}\left(2^{j} x-y\right) \mathrm{d} y\right|=\left|\dot{\Delta}_{0} f_{j}\left(2^{j} x\right)\right|.
$$
Moreover, we get by \eqref{eq-local-1} that for each $p\geq1,$
$$
\big\|\dot{\Delta}_{0} f_{j}\big\|_{L^{\infty}\left(\mathbb{R}^{d}\right)} \leq C\left\|f_{j}\right\|_{1,1}=C\left\|f\left(\cdot / 2^{j}\right)\right\|_{1,1} \leq C\left\|f\left(\cdot / 2^{j}\right)\right\|_{p, 1},
$$
namely,
$$
\big\|\dot{\Delta}_{j} f\big\|_{L^{\infty}\left(\mathbb{R}^{d}\right)} \leq C\left\|f\left(\cdot / 2^{j}\right)\right\|_{p, 1}  \leq C 2^{d j / p}\|f\|_{p, 2^{-j}}.
$$
where we have used the following fact
\[\left(\int_{|x-y| \leq 1}\left|f\left(\frac{y}{2^{j}}\right)\right|^{p} \mathrm{~d} y\right)^{\frac1p}=\left(2^{d j} \int_{\left|2^{-j} x-y\right| \leq 2^{-j}}|f(y)|^{p} \mathrm{~d} y\right)^{\frac1p} \leq C 2^{d j / p}\|f\|_{p, 2^{-j}}.\]
In the same way, we can show estimate
\[\big\|\dot{S}_{j} f\big\|_{L^{\infty}\left(\mathbb{R}^{d}\right)}   \leq C 2^{d j / p}\|f\|_{p, 2^{-j}},\] and we complete the proof of the lemma.
\end{proof}
With Lemma \ref{GYong} in hand, we begin to show Proposition \ref{prop-bound} which is the key estimate in our paper.
\begin{proof}[Proof of Proposition \ref{prop-bound}]
From system \eqref{che-M}, we can write in term of semigroup that
\begin{align*}
\nabla c(t) =&e^{\frac{t}{\tau}(-1+\Delta)}\nabla\big(\psi(x/M) c_0\big)+\int_{0}^te^{\frac{(t-s)}{\tau}(-1+\Delta)} \nabla n \dd s.
\end{align*}
According to low-high frequency decomposition, we split $\nabla c$ into two parts as follow
\[\nabla c=\dot{S}_0 \nabla c+\sum_{j\geq 0}\dot{\Delta}_j \nabla c:=\nabla c^L+\nabla c^H.\]
For low frequency regime,   by the generalized Young inequality \eqref{GY-lmf}, we have
\[\|\nabla c^L\|_{L^\infty(\RR^d)}\leq C\|\nabla c_0\|_{L^2_{\uloc}(\RR^d)}.\]
As for high frequency regime, we can bound it by using \cite[Lemma 2.4]{BCD} and \eqref{GY-lmf} that for $k>d,$
\begin{align*}
\left\| \nabla c^H\right\|_{L^\infty(\RR^d)}\leq& C\|c_0\|_{W^{1,\infty}(\RR^d)}+\sum_{j\geq 0}\int_{0}^t\left\|\dot{\Delta}_je^{\frac{(t-s)}{\tau}(-1+\Delta)} \nabla n \right\|_{L^\infty(\RR^d)}\dd s\\
\leq &C\|c_0\|_{W^{1,\infty}(\RR^d)}+C\sum_{j\geq 0}C2^{j(1+\frac{d}{k})}\int_{0}^te^{-c(t-s)2^{2j}}\left\|  n(s) \right\|_{L^{k}_{\uloc}(\RR^d)}\dd s\\
\leq &C\|c_0\|_{W^{1,\infty}(\RR^d)}+C\|  n  \|_{L^{\infty}_TL^{k}_{\uloc}(\RR^d)}\sum_{j\geq 0}C2^{j(-1+\frac{d}{k})} \\
\leq &C\|c_0\|_{W^{1,\infty}(\RR^d)}+C\|  n  \|_{L^{\infty}_TL^{k}_{\uloc}(\RR^d)}.
\end{align*}
Combining the above both estimates yields
\begin{align*}
\|\nabla c\|_{L^\infty(\RR^d)}\leq C\|\nabla c_0\|_{L^2_{\uloc}(\RR^d)}+C\|c_0\|_{W^{1,\infty}(\RR^d)} +C\|  n  \|_{L^{\infty}_TL^{k}_{\uloc}(\RR^d)}.
\end{align*}
Similarly, we can show that $k>d,$
\begin{align*}
\|n(t)\|_{L^\infty(\RR^d)}\leq &C\|n\|_{L^1_{\uloc}(\RR^d)}+C\|c_0\|_{L^{\infty}(\RR^d)}+\sum_{j\geq 0}\int_{0}^t\left\|\dot{\Delta}_je^{(t-s)\Delta} \nabla (n \nabla c)\right\|_{L^\infty(\RR^d)}\dd s\\
&+\mu\sum_{j\geq 0}\int_{0}^t\left\|\dot{\Delta}_je^{(t-s)\Delta}   n^2 \right\|_{L^\infty(\RR^d)}\dd s\\
\leq &C\|n\|_{L^1_{\uloc}(\RR^d)}+C\|c_0\|_{L^{\infty}(\RR^d)}+C\mu\sum_{j\geq 0}2^{j  \frac{2d}{k} }\int_{0}^t e^{-c(t-s)2^{2j}}\|n\|^2_{L^{k}_{\uloc}(\RR^d)}\dd s\\
&+C\sum_{j\geq 0}2^{j(1+\frac{d}{k})}\int_{0}^te^{-c(t-s)2^{2j}}\|n\|_{L^k_{\uloc}(\RR^d)}\|\nabla c\|_{L^{\infty}(\RR^d)}\dd s\\
\leq &C\left(\|n\|_{L^1_{\uloc}(\RR^d)}+\|c_0\|_{L^{\infty}(\RR^d)}+\mu \|n\|^2_{L^\infty_TL^k_{\uloc}(\RR^d)}
 +\|\nabla c\|^2_{L^\infty_TL^{\infty}(\RR^d)}\right).
\end{align*}
It remains  to show the uniformly local $L^2_{\uloc}(\RR^d)$-estimate for $c$. In the same fashion as above for $\nabla c$, one has
\begin{align*}
\frac{\tau}{2}\frac{\textnormal{d}}{\textnormal{d}t}\int_{\RR^d}\left(  \phi^R_{x_0}c\right)^2 \dd x+ \int_{\RR^d} \left(\phi^R_{x_0}c\right)^2 \dd x-\int_{\RR^d}\Delta cc \left(\phi^R_{x_0}\right)^2 \dd x= \int_{\RR^d}nc \left(\phi^R_{x_0}\right)^2 \dd x.
\end{align*}
Integrating by parts
\begin{align*}
-\int_{\RR^d} \Delta cc\left(\phi^R_{x_0} \right)^2\dd x=\int_{\RR^d} \left(\phi^R_{x_0}\nabla c\right)^2 \dd x+2\int_{\RR^d}  \phi^R_{x_0}c\nabla c \cdot\nabla\phi^R_{x_0} \dd x.
\end{align*}
By the Cauchy-Schwarz inequality, we see that
\begin{equation*}
2\int_{\RR^d}  \phi^R_{x_0}c\nabla c \cdot\nabla\phi^R_{x_0} \dd x\leq
 2C\left\|\phi^R_{x_0}c\right\|_{L^2(\RR^d)} \left\|\nabla c\right\|_{L^2_{\uloc}(\RR^d)}
\end{equation*}
On the other hand, we get by the H\"older inequality that
\begin{align*}
\int_{\RR^d}nc \left(\phi^R_{x_0}\right)^2 \dd x\leq& \left\|\phi^R_{x_0} \right\|_{L^2(\RR^d)}\|n\|_{L^\infty(\RR^d)}\left\|\phi^R_{x_0}c\right\|_{L^2(\RR^d)}\\
\leq&CR^{\frac{d}{2}} \|n\|_{L^\infty(\RR^d)}\left\|\phi^R_{x_0}c\right\|_{L^2(\RR^d)}.
\end{align*}
Collecting the above estimates yields
\begin{align*}
 \tau \frac{\textnormal{d}}{\textnormal{d}t}\left\|\phi^R_{x_0}c\right\|_{L^2(\RR^d)}+ \left\|\phi^R_{x_0}c\right\|_{L^2(\RR^d)}\leq C\left\|\nabla c\right\|_{L^2_{\uloc}(\RR^d)}  +CR^{\frac{d}{2}} \|n\|_{L^\infty(\RR^d)}.
\end{align*}
Integrating the above inequality with respect to time $t$, we readily obtain
\begin{align*}
\left\|\phi^R_{x_0}c(t)\right\|_{L^2(\RR^d)}\leq& e^{-\frac{t}{\tau}} \left\|\phi^R_{x_0}\psi(x/M)c_0\right\|_{L^2(\RR^d)}\\&+\left(C\left\|\nabla c\right\|_{L^\infty_T L^2_{\uloc}(\RR^d)}  +CR^{\frac{d}{2}} \|n\|_{L^\infty_TL^\infty(\RR^d)}\right)\int_0^t \frac{1}{\tau}e^{-\frac{1}{\tau}(t-s)}\dd s.
\end{align*}
Since \[\left\|\nabla c\right\|_{L^\infty_T L^2_{\uloc}(\RR^d)}  +  \|n\|_{L^\infty_TL^\infty(\RR^d)}\leq C(\lambda, \tau, \mu, \chi, d, R), \]
we immediately have
\[\|c\|_{L_T^\infty L^2_{\uloc}(\RR^d)} \leq\|c_0\|_{L^2_{\uloc}(\RR^d)}  +C(\lambda, \tau, \mu, \chi, d, R).\]
By the sharp interpolation inequality, we have
\begin{equation*}
\| c\|_{L^\infty(\RR^d)}\leq C\|c\|^{\frac{2}{d+2}}_{L^2_{\uloc}(\RR^d)}\|\nabla c\|^{\frac{d}{d+2}}_{L^\infty(\RR^d)}.
\end{equation*}
Collecting the above estimate, we eventually get that for $0<t<T_{\max},$
\[\|n(t)\|_{L^\infty(\RR^d)}+\|  c(t)\|_{W^{1,\infty}(\RR^d)}\le C\left(\tau,\chi, k, \lambda, d,  \mu,\|n_0\|_{L^\infty(\RR^d)},\|c_0\|_{W^{1,\infty}(\RR^d)}\right), \]
and then we finish proof of the proposition.
\end{proof}
In terms of Proposition \ref{prop-bound} and continuation criterion \eqref{blowup} in Theorem \ref{thm-local},  we know that $T_{\max}=\infty,$ that is, the couple $(n^M,C^M)$ is global-in-time solution to problem \eqref{che-M}-\eqref{che-I-M} satisfying
\[\left\|n^M(t)\right\|_{L^\infty(\RR^d)}+\left\|  c^M(t)\right\|_{W^{1,\infty}(\RR^d)}\le C\left(\tau,\chi, k, \lambda, d,  \mu,\|n_0\|_{L^\infty(\RR^d)},\|c_0\|_{W^{1,\infty}(\RR^d)}\right) \] for all $t\geq0.$ Furthermore,  we can conclude by the compactness argument argument as used in \cite{ZZ} that  there exists a couple $(n,u)$ which the limit of family $(n^M,C^M)$ such that \[\left\|n(t)\right\|_{L^\infty(\RR^d)}+\left\|  c(t)\right\|_{W^{1,\infty}(\RR^d)}\le C\left(\tau,\chi, k, \lambda, d,  \mu,\|n_0\|_{L^\infty(\RR^d)},\|c_0\|_{W^{1,\infty}(\RR^d)}\right) \] for all $t\geq0,$ and it is global-in-time solution to problem \eqref{che}-\eqref{che-I}.

Our task is now to show the uniqueness of solution $(n,c)$ to problem \eqref{che}-\eqref{che-I}. Indeed, we can get it in uniformly local $L^2_{\uloc}(\RR^d)$ framework by   modifying the proof of uniqueness in Theorem \ref{thm-local}.  Thus we complete the proof of Theorem \ref{result}.
\section*{Acknowledgement}
Yao Nie and X. Zheng was partially supported by the National Natural Science Foundation of China under grant No. 11871087.


\begin{thebibliography}{99}
\bibitem{BCD} H. \textsc{Bahouri}, J.-Y. \textsc{Chemin}, R. \textsc{Danchin}, Fourier Analysis and Nonlinear Partial Differential Equations, Springer-Verlag, Berlin, 2011.
\bibitem{BWS06} M. D. \textsc{Baker}, P. M. \textsc{Wolanin}, J. B. \textsc{Stock}, \emph{Signal transduction in bacterial chemotaxis}, Bioessays, 28(1)(2006), 9-22.
\bibitem{CC08}
V. \textsc{Calvez}, L. \textsc{Corrias}, \emph{The parabolic-parabolic Keller-Segel model in $\RR^2$}, Commun. Math. Sci., 6(2)(2008),417-447.
\bibitem{Cao15}
 X. \textsc{Cao}, \emph{Global bounded solutions of the higher-dimensional Keller-Segel system
under smallness conditions in optimal spaces}, Discrete Contin. Dynam. Syst. Ser. A 35(5)(2015) 1891-1904.
\bibitem{CP08}
 L. \textsc{Corrias}, B. \textsc{Perthame}, \emph{Asymptotic decay for the solutions of the parabolic-parabolic Keller-Segel chemotaxis system in critical spaces}, Math. Comput. Modelling 47 (2008), 755-764.
\bibitem{CSS05}
J. \textsc{Condeelis} ,  R. H. \textsc{Singer},  J. E. \textsc{Segall}, \emph{The great escape: when cancer cells hijack the genes for chemotaxis and motility}, Annu. Rev. Cell Dev. Biol. 21(2005), 695-718.
\bibitem{DW06}
D. \textsc{Dormann},  C. J. \textsc{Weijer}, \emph{Chemotactic cell movement during Dictyostelium development and gastrulation}, Curr. Opin. Genet. Dev. 16(4)(2006), 367-373.
\bibitem{FIWY16}
K. \textsc{Fujie}, A. \textsc{Ito}, M. \textsc{Winkler},  T. \textsc{Yokota}, \emph{Stabilization in a chemotaxis model for tumor invasion}, Discrete Contin. Dyn. Syst. A 36 (2016), 151-169.
\bibitem{GZ98}
 H. \textsc{Gajewski}, K. \textsc{Zacharias}, \emph{Global behavior of a reaction-diffusion system modelling chemotaxis}, Math. Nachr. 195 (1998), 77-114.
\bibitem{HP04}
T. \textsc{Hillen},  A. \textsc{Potapov}, \emph{The one-dimensional chemotaxis model: Global existence and asymptotic profile}, Math. Methods Appl. Sci. 27, 1783-1801 (2004).
\bibitem{HK05}
 E. \textsc{Hildebrand}, U. B. \textsc{Kaupp}, \emph{Sperm chemotaxis: a primer}, Ann. N. Y. Acad. Sci. 1061(2005), 221-225.
\bibitem{HW01}
D. \textsc{Horstmann}, G. \textsc{Wang}, \emph{Blow-up in a chemotaxis model without symmetry assumptions}, European J. Appl. Math. 12 (2)(2001) 159-177.
\bibitem{KS70}
E. F. \textsc{Keller} , L. A.\textsc{Segel}, \emph{Initiation of slime mold aggregation viewed as an instability}, J. Theor. Biol. 26(1970), 399-415.
\bibitem{KS71}
E. F. \textsc{Keller}, L. A.\textsc{Segel}, \emph{Model for chemotaxis}, J. Theor. Biol. 30(1971), 225-234.
\bibitem{MT96}
M. \textsc{Mimura},   T. \textsc{Tsujikawa}, \emph{Aggregating pattern dynamics in a chemotaxis model including growth}, Phys. A, 230(3-4)(1996),  449-543.
\bibitem{NSY97}
 T. \textsc{Nagai}, T. \textsc{Senba}, K. \textsc{Yoshida}, \emph{Application of the Trudinger-Moser inequality to a parabolic system of chemotaxis}, Funkcial. Ekvac. Ser. Int. 40 (1997) 411-433.
\bibitem{Pa53}
C. S. \textsc{Patlak}, \emph{Random walk with persistence and external bias}, Bull. Math. Biophys. 15(1953), 311-338.
\bibitem{OT02}
K. \textsc{Osaki}, T. \textsc{Tsujikawa}, A. \textsc{Yagi}, M. \textsc{Mimura}, \emph{Exponential attractor for a chemotaxis-growth system of equations}, Nonlinear Anal. Theory, Methods \& Applications 51(1)(2002), 119-144.
\bibitem{OY01}
K. \textsc{Osaki}, A. \textsc{Yagi}, \emph{Finite dimensional attractor for one-dimensional Keller-Segel equations}, Funkcialaj Ekvacioj, 44(2001), 441-469.
\bibitem{OY02}
K. \textsc{Osaki} , A. \textsc{Yagi}, \emph{Global existence for a chemotaxis-growth system in
$\RR^2$}, Adv. Math. Sci. Appl. 12(2002), 587-606.
\bibitem{SS01}
T. \textsc{Senba}, T. \textsc{Suzuki}, \emph{Parabolic system of chemotaxis: blowup in a finite and the infinite time}, Methods Appl. Anal. 8(2)(2001),349-367.
\bibitem{TW15}
Y. \textsc{Tao},  M. \textsc{Winkler}, \emph{Boundedness and decay enforced by quadratic degradation in a threedimensional chemotaxis-fluid system}, Z. Angew. Math. Phys., 66 (2015), 2555-2573.
\bibitem{Wu05}
 D. \textsc{Wu}, \emph{Signaling mechanisms for regulation of chemotaxis}, Cell Res. 15(1)(2005),
 52-56.
\bibitem{Win10JDE}
M. \textsc{Winkler}, \emph{Aggregation vs. global diffusive behaviour in the higher-dimensional Keller-Segel model}, J. Differential Equations 248 (2010) 2889-2905.
\bibitem{Win10}
M. \textsc{Winkler}, \emph{Boundedness in the Higher-Dimensional Parabolic-Parabolic Chemotaxis System with Logistic Source},  Communications in Partial Differential Equations, 35(8)(2010),1516-1537.
\bibitem{Win13}
M. \textsc{Winkler}, \emph{Finite-time blow-up in the higher-dimensional parabolic-parabolic Keller-Segelsystem}, J. Math. Pure Appl. 100(5)(2013), 748-767.
\bibitem{Win14}
M. \textsc{Winkler}, \emph{Global asymptotic stability of constant equilibria in a fully parabolic chemotaxis system with strong logistic dampening}, J. Differential Equations, 257 (2014), 1056-1077.
\bibitem{Win20}
M. \textsc{Winkler}, \emph{Single-point blow-up in the Cauchy problem for the higher-dimensional Keller-Segel system}, Nonlinearity, 33 (2020) 5007-5048.
\bibitem{Xiang15}
T. \textsc{Xiang}, \emph{Boundedness and global existence in the higher-dimensional parabolic-parabolic
chemotaxis system with/without growth source}, J. Differential Equations, 258 (2015),4275-4323.
\bibitem{Xi18}
T. \textsc{Xiang}, \emph{Chemotactic Aggregation versus Logistic Damping on Boundedness in the 3D Minimal Keller-Segel Model}, SIAM J. Appl. Math., 78(5)(2018), 2420-2438.
\bibitem{Xiang18}
T. \textsc{Xiang}, \emph{How strong a logistic damping can prevent blow-up for the minimal Keller-Segel
chemotaxis system?}, J. Math. Anal. Appl., 459 (2018),  1172-1200.
\bibitem{YCJZ15}
 C. \textsc{Yang}, X. \textsc{Cao}, Z. \textsc{Jiang},  S. \textsc{Zheng}, \emph{Boundedness in a quasilinear fully parabolic Keller-Segel system of higher dimension with logistic source}, J. Math. Anal. Appl., 430(2015), 585-591.
 \bibitem{ZZ}
 Q. \textsc{Zhang}, X. \textsc{Zheng},\emph{Global well-posedness for the two-dimensional incompressible chemotaxis-Navier-Stokes equations}, SIAM J. Math. Anal. 46 (2014), no. 4, 3078-3105.
\end{thebibliography}
\end{document}